\documentclass[12pt]{article}
\usepackage{color}
\usepackage{amsmath}

\usepackage{extarrows}
\usepackage{geometry}
\usepackage{authblk} 
\usepackage{hyperref} 
\usepackage{color}
\usepackage{ntheorem}
\usepackage{enumitem}
\usepackage{graphicx}
\usepackage{subcaption}

\usepackage[mathscr]{euscript}
\usepackage[utf8]{inputenc}
\def\q{\hfill\rule{1ex}{1ex}}
\def\0{\emptyset}

\def\q{\hfill\rule{1ex}{1ex}}
\def\QEDopen{{\hfill\setlength{\fboxsep}{0pt}\setlength{\fboxrule}{0.2pt}\fbox{\rule[0pt]{0pt}{1.3ex}\rule[0pt]{1.3ex}{0pt}}}}

\newtheorem{theorem}{Theorem}[section]

\newtheorem{lemma}[theorem]{Lemma}
\newtheorem{claim}[theorem]{Claim}

\newtheorem{conjecture}[theorem]{Conjecture}

\newenvironment{provekplF=1}{{\noindent\it Proof of Theorem  \ref{thm: when beta F=1}}.}{\hfill $\square$\par}
\newenvironment{proof}{{\noindent\it Proof.}}{\hfill $\square$\par}

\usepackage{graphicx}

\usepackage{enumerate}
\usepackage{enumitem}

\usepackage{
	amsmath,			
	amssymb,			
	enumerate,		    
	graphicx,			
	lastpage,			
	multicol,			
	multirow,			
	pifont,			    
}

\usepackage[numbers]{natbib}

\newcommand{\ber}{\text{Berge-}}

\newcommand{\mh}{\mathcal{H}}
\newcommand{\mg}{\mathcal{G}}
\newcommand{\lf}{\left\lfloor}
\newcommand{\rf}{\right\rfloor}

\newcommand{\lchuto}{\lf\frac{\ell+1}{2}\rf}

\newcounter{cases}
\newcounter{subcases}[cases]

\begin{document}


\title{Extremal results on Berge disjoint paths}
\author{
    {\small\bf Xiamiao Zhao}\thanks{email:  zxm23@mails.tsinghua.edu.cn}\quad
    {\small\bf Yiyan Zhan}\thanks{email:  zhanyy24@mails.tsinghua.edu.cn}\quad
    {\small\bf Mei Lu}\thanks{email: lumei@tsinghua.edu.cn}\\
    {\small Department of Mathematical Sciences, Tsinghua University, Beijing 100084, China.}\\
}

\date{}

\maketitle\baselineskip 16.3pt

\begin{abstract}
The well-known Erd\H{o}s-Gallai Theorem gave the Tur\'an number of paths. Bushaw and Kettle generalized this result to consider the Tur\'an number of disjoint paths. Since then, many studies are focused on the Tur\'an number of linear forest.
For a graph $F$, an $r$-uniform hypergraph $\mathcal{H}$ is a $\text{Berge-} F$ if there is a bijection $\phi: E(F)\to E(\mathcal{H})$ such that $e\subseteq \phi(e)$ for each $e\in E(F)$. When $F$ is a path, we call $\text{Berge-} F$ a Berge path.
The Tur\'an number of Berge paths was initially studied by Gy\H{o}ri, Katona and Lemons. They gave the value of $\text{ex}_r(n,\text{Berge-}P_\ell)$ for $\ell>r+1$. This result is a generalization of Erd\H{o}s-Galli Theorem. 
Since then, the Tur\'an number of Berge paths has received widespread attention.

Recently, Zhou, Gerbner and Yuan initially studied the Tur\'an number of Berge disjoint paths and for the cases when all the paths have odd length.
In this paper, we give a more general result, which gives the exact value of $\mathrm{ex}_r(n,\text{Berge-} kP_{\ell})$ for all $k\geq 2$, $r\ge 3$, and $\ell\geq r+7$. 
\end{abstract}


\vskip.3cm
\section{Introduction}
A hypergraph $\mh=(V(\mh),E(\mh))$ is consists of a vertex set $V(\mh)$ and a hyperedge set $E(\mh)=\{E\subseteq V(\mh)|E\neq\emptyset\}$. If $|E|=r$ for every $E\in E(\mh)$, then $\mh$ is called an $r$-uniform hypergraph or an $r$-graph. In general, we assume $V(\mh)=\cup_{E\in E(\mh)}E$ and we also use $\mh$ to denote the hyperedge set of $\mh$.
\par
Let $\mathcal{F}$ be a family of $r$-graphs. An $r$-graph $\mathcal{H}$ is called $\mathcal{F}$-free if $\mh$ does not contain any member in $\mathcal{F}$ as a subhypergraph. The Tur\'an number $\mathrm{ex}_{r}(n,\mathcal{F})$ is the maximum number of hyperedges in an $\mathcal{F}$-free $r$-graph on $n$ vertices. When $r=2$, we write $ex(n,\mathcal{F})$ instead of $\mathrm{ex}_{2}(n,\mathcal{F})$.
\par
Let $F=(V(F),E(F))$ be a graph. An $r$-graph $\mathcal{H}$ is a $\ber F$ if there is a bijection $\phi:E(F)\to E(\mh)$ such that $e \subseteq \phi(e)$ for each $e\in E(F)$. For a fixed $F$, there are many hypergraphs that are a $\ber F$, for convenience, we refer to this collection of hypergraphs as $\ber F$.
More specifically, if $V(F)=\{v_1,\ldots,v_s\}$, then for any copy of $\ber F$, denoted as $\mh$, there is $\{v_{1}',\ldots,v_{s}'\}\subseteq V(\mh)$ such that for all $v_i v_j\in E(F)$, $\{v_i',v_j'\}\subseteq \phi(v_i v_j)$. The vertices in $\{v_1',\dots v_s'\}$ are called the \emph{defining vertices} of $\ber F$, and edges in $\{\phi(E):E\in E(F)\}$ are called the \emph{defining hyperedges} of $\ber F$. For $F$ and a $\ber F$ $\mh$, the defining vertices of $\mh$ might not be unique. By ``the defining vertices of $\mh$" we mean one group of these vertices without additional explanation, the same for ``the defining hyperedges of $\mh$". When $F$ is a path, we call $\ber F$ as \emph{Berge path}.
 Berge \cite{berge1984hypergraphs} defined the Berge cycle, and Gy\H{o}ri, Kantona and Lemons \cite{gyHori2016hypergraph} defined the Berge path. Later, Gerbner and Palmer \cite{gerbner2017extremal} generalized the established concept of Berge path to general graphs.

 The (hypergraph) Tur\'an problems are the central topics of extremal combinatorics. The  ($\ber$type) problems about paths have been studied extensively. Let $P_{\ell}$ denote a path with $\ell$ edges, and $kP_{\ell}$ denote $k$ disjoint paths with length $\ell$. Erd\H{o}s and Gallai \cite{erdos1959maximal} gave the following famous conclusion.
\begin{theorem}[\cite{erdos1959maximal}]\label{thm: Pl}
    For any $n,\ell\in\mathbb{N}$, $ex(n,P_{\ell})\leq \frac{\ell-1}{2}n$.
\end{theorem}
Bushaw and  Kettle \cite{bushaw2011turan} considered generalization of the Erd\H{o}s-Gallai Theorem about disjoint paths and obtained the following result.
\begin{theorem}[\cite{bushaw2011turan}]\label{thm:kPl}
For $k\geq 2,\ell\geq 3$, and $n\geq 2(\ell+1)+2k(\ell+1)\left(\lceil\frac{\ell+1}{2}\rceil+1\right)\binom{\ell+1}{\lfloor\frac{\ell+1}{2}\rfloor}$,
\begin{equation*}
    ex(n,kP_{\ell})=\left(k\lf\frac{\ell+1}{2}\rf-1\right)\left(n-k\lf\frac{\ell+1}{2}\rf+1\right)+\binom{k\lf\frac{\ell+1}{2}\rf-1}{2}+c_{\ell},
\end{equation*}
where $c_\ell=1$ if $\ell$ is even and $c_{\ell}=0$ if $\ell$ is odd.
\end{theorem}
A forest is called \emph{linear forest} if all its components are paths, specifically, $kP_\ell$ is a linear forest. Many studies are focused on the Tur\'an number of linear forest, see for example \cite{bushaw2011turan,lidicky2012turan,yuan2021turan, Zhang2025TuranTP}.

Gy\H{o}ri, Katona and Lemous \cite{gyHori2016hypergraph} generalized the Erd\H{o}s-Gallai theorem to Berge paths in hypergraphs. Specifically they determined $\mathrm{ex}_{r}(n,\ber P_{\ell})$ for the case when $\ell>r+1>3$ and the case when $r\geq\ell>2$. The case when $\ell=r+1>2$ was settled by Davoodi, Gy\H{o}ri, Methuku and Tompkins
\cite{DavoodiGyoriMethukuTompkins2018}.
\begin{theorem}[\cite{gyHori2016hypergraph},
\cite{DavoodiGyoriMethukuTompkins2018}]
\label{thm: turan of BPl}

    (1) If $\ell\geq r+1>3$, then $\mathrm{ex}_r(n,\ber P_\ell)\leq \frac{n}{\ell}{\ell\choose r}$. Furthermore, this bound is sharp whenever $\ell$ divides $n$.

    \noindent(2) If $r\geq \ell>2$, then $\mathrm{ex}_r(n,\ber P_\ell)\leq \frac{n(\ell-1)}{r+1}$. Furthermore, this bound is sharp whenever $r+1$ divides $n$.
\end{theorem}
Let $\mathrm{ex}_r^{con}(n,\mathcal{F})$ denote the maximum number of hyperedges in a connected $ r$-graph on $n$ vertices that does not contain any copy of $\mathcal{F}$ as a sub-hypergraph.
Gy\H{o}ri, Salia and Zamora \cite{gyHori2021connected} determined the value of $\mathrm{ex}_r^{con}(n,\ber P_\ell)$.
\begin{theorem}[\cite{gyHori2021connected}]\label{thm: connect Pl}
    For all integers $n,\ell$ and $r$, there exists $N_{\ell,r}$ such that if $n>N_{\ell,r}$ and $\ell\geq 2r+13\geq 18,$ then
    $$\mathrm{ex}_r^{con}(n,\ber P_\ell)=\binom{\ell'-1}{r-1}(n-\ell'+1)+\binom{\ell'-1}{r}+\mathbb{I}_\ell\cdot\binom{\ell'-1}{r-2},$$
    where $\ell'=\lf\frac{\ell+1}{2}\rf$, $\mathbb{I}_\ell=1$ if $\ell$ is even and $\mathbb{I}_\ell=0$ if $\ell$ is odd.
\end{theorem}
Recently, 
Zhou, Gerbner and Yuan \cite{Zhou2025OnTP} gave the Tur\'an number of the Berge two paths in some cases.
\begin{theorem}[\cite{Zhou2025OnTP}]\label{thm: two berge path}
    Let $r\geq 3$ and $n$ be sufficiently large.
    \par\noindent
    (i) If $\ell$ is an odd integer and $\ell\geq 2r+11$, then
    $$\mathrm{ex}_{r}(n,\ber P_{\ell}\cup P_{1})=\max\left\{\mathrm{ex}_{r}(n,\ber P_{\ell}),\binom{\frac{\ell+1}{2}}{r-1}(n-\frac{\ell-1}{2})+\binom{\frac{\ell+1}{2}}{r}\right\}.$$
    \par\noindent
    (ii) If $\ell_1,\ell_2$ are odd integers and $\ell_1\geq \ell_2\geq r+6$, then
    $$\mathrm{ex}_{r}(n,\ber P_{\ell_1}\cup P_{\ell_2})=\max\left\{\mathrm{ex}_{r}(n,\ber P_{\ell_1}),\binom{\frac{\ell_1+\ell_2}{2}}{r-1}(n-\frac{\ell_1+\ell_2}{2})+\binom{\frac{\ell_1+\ell_2}{2}}{r}\right\}.$$
\end{theorem}
\par\noindent
Note that by Theorem \ref{thm: turan of BPl}, when $n$ is large enough, we have $\mathrm{ex}_{r}(n,\ber P_{\ell_1})<\binom{\frac{\ell_1+\ell_2}{2}}{r-1}(n-\frac{\ell_1+\ell_2}{2})+\binom{\frac{\ell_1+\ell_2}{2}}{r}$. Consequently Theorem \ref{thm: two berge path} (ii)  actually gave the value of $\mathrm{ex}_r(n,\ber 2P_\ell)$ when $\ell \geq r+6$ and $\ell$ is odd.
Let $M_k$ be a matching with $k$ edges. Khormali and Palmer \cite{KHORMALI2022103506} completely determined the value of $\mathrm{ex}_r(n,\ber M_k)$ for $n$ large enough in the case that $k\geq 1$ and $r\geq 2$.

Let $T_{\ell}$ and $S_{\ell}$ be a tree  and  a star with $\ell$ edges, respectively. In \cite{Zhou2025OnTP},  Zhou, Gerbner and Yuan detemined $\mathrm{ex}_{r}(n,\ber T_{\ell}\cup (k-1)S_i)$ and $\mathrm{ex}_{r}(n,\ber F\cup M_{k-1})$ under some conditions and restrictions. In the same paper, they also give bounds for $\mathrm{ex}_{r}(n,\ber P_{\ell_1}\cup (k-1)S_{\ell_2})$ and $\mathrm{ex}_{r}(n,\ber T_{\ell_1}\cup (k-1)S_{\ell_2})$ in some cases. A hypergraph is said to be linear if every two distinct hyperedges intersect in at most one vertex. Győri and Salia \cite{GYORI202536} considered Berge path problems in linear $3\text{-}$graph and proved that the number of hyperedges of a linear $3\text{-}$graph contained no $\ber P_{k}$ for all $k\geq 4$ is at most $\frac{k-1}{6}n$.

    Zhou, Gerbner and Yuan \cite{Zhou2025OnTP} also determined $\mathrm{ex}_{r}^{con}(n,\ber P_{\ell_1}\cup\ldots\cup P_{\ell_k})$ when $\ell_{i}$ are all odd integers. Many other studies about the Tur\'an number of Berge copies of graphs are listed here \cite{gerbner2024turan,gerbner2020general,gerbner2019asymptotics,ghosh2024book,gyHori2006triangle,gyHori2012hypergraphs}.

In this paper, we determine the exact value of $\mathrm{ex}_r(n,\ber kP_\ell$ for all $k\geq 2$ , $r\geq 3$ and $\ell\geq 2r+7.$ This generalize the result of Zhou, Gerbner and Yuan \cite{Zhou2025OnTP}. The constant $7$ appears has a similar reason as the constant $13$ in Theorem \ref{thm: connect Pl}. We use different method to solove the case when $k=2$ and when $k\geq 3$. Thus we deal with the different value of $k$ respectively.
\begin{theorem}\label{thm: berge linear forest} Let $\ell,r,k$ be integers with $k\geq 2$, $r\ge 3$, $\ell'\geq r$ and $2\ell'\geq r+7$, where $\ell'=\lchuto$.
Then    there exist an $N_{\ell,r,k}$ such that for $n>N_{\ell,r,k}$,
    $$\mathrm{ex}_r(n,\ber kP_{\ell})={k\ell'-1\choose r-1}(n-k\ell'+1)+{k\ell'-1\choose r}+\mathbb{I}_\ell\cdot{k\ell'-1\choose r-2},$$
    where  $\mathbb{I}_\ell=1$ if $\ell$ is even and $\mathbb{I}_\ell=0$ if $\ell$ is odd.
\end{theorem}


First, we construct two hypergraphs to obtain the lower bounds of both Theorem \ref{thm: berge linear forest}.
    Let $\mh_0(n,k,\ell)$ and $\mh_1(n,k,\ell)$ be two $r$-uniform hypergraphs with $n$ vertices,
where $V(\mh_0(n,k,\ell))=V(\mh_1(n,k,\ell))=A\cup B$  with  $|A|=k\ell'-1$ and $|B|=n-k\ell'+1$. 
 Let $$\mathcal{E}_1=\left\{E\in\binom{V(\mh)}{r}:E\subseteq A\right\},~~\mathcal{E}_2=\left\{E\in\binom{V(\mh)}{r}:|E\cap B|=1\right\},$$ and fix two vertices $u,v\in B$, set $$\mathcal{E}_3=\left\{E\in \binom{V(\mh)}{r}: E\cap B=\{u,v\} \right\}.$$
If $\ell$ is odd, then set $\mh_0(n,k,\ell)=\mathcal{E}_1\cup \mathcal{E}_2$. If $\ell$ is even, set $\mh_1(n,k,\ell)=\mathcal{E}_1\cup \mathcal{E}_2\cup \mathcal{E}_3$. It is easy to check that $\mh_{\mathbb{I}_{\ell}}(n,k,\ell)$ is $\ber kP_\ell$-free, where $\mathbb{I}_{\ell}=1$ if $\ell$ is even and $\mathbb{I}_{\ell}=0$ if $\ell$ is odd. Then
\begin{align}\label{3-1}
\mathrm{ex}_r(n,\ber kP_{\ell})\ge {k\ell'-1\choose r-1}(n-k\ell'+1)+{k\ell'-1\choose r}+\mathbb{I}_\ell\cdot\binom{k\ell'-1}{r-2}.
\end{align}

 For the upper bound, we first deal with the case when $k=2$, where we concentrate on the longest Berge path and large Berge cycles in the hypergraph to find the upper bound. Then based on this result we make an induction for $k\geq 3$. The sketch of the proof for the case is to find a set of vertices that intersect with almost all the hyperedges, and the key Lemma \ref{lem: A is bounded} shows this set can be bounded by a constant. Then with further analysis, we determine the exact value of $\mathrm{ex}_r(n,\ber k P_\ell)$.

The paper is organized as follows. In Section 2, we will give some notations and inequalities. Section 3 gives the proof of Theorem \ref{thm: berge linear forest} when $k=2$. Section 4 gives the  completed proof of  Theorem \ref{thm: berge linear forest}. We conclude the paper by two conjectures in Section 5.
\section{Preliminary}
In this section, we define some notations and provide some inequalities which we will use later.
Let  $F$ be a graph and  $\mh$  hypergraph contained a copy of $\ber F$. Denote $DV_\mh (\ber F)$  and $DE_\mh(\ber F)$ the defining vertices and  the defining hyperedges of  $\ber F$ in $\mh$, respectively.


Let $\mathcal{H}$ be an $r$-uniform hypergraph and $V_0\subseteq V(\mathcal{H})$. We denote $\mathcal{H}[V_0]=\{E\in E(\mathcal{H}): E\subseteq V_0\}$.
   For a set  $S\subseteq V(\mathcal{H})$,   denote $\mh[V(\mh)\setminus S]$ by $\mh-S$. For short, we set $\mh-v=\mh-\{v\}$. For $\mathcal{E}\subseteq \mathcal{H}$, let   $\mh\setminus \mathcal{E}$ denote the hypergraph obtained by deleting hyperedges in $\mathcal{E}$ and  the resulting isolated vertices from $\mh$. And for a set $S\subseteq V(\mh)$ with $|S|\leq r-1$, let $\mathcal{N}_\mh(S)=\{E\setminus S: S\subseteq E \text{~and ~} E\in \mh\}$. We write $\mathcal{N}_\mh(v)$ instead of $\mathcal{N}_\mh(\{v\})$ for convenience. And let $N_\mh(v)=\{u\in V(\mh):\mbox{there is~} E\in \mh \mbox{~such that~} \{u,v\}\subseteq E\}$.

 For  $X,Y\subseteq V(\mh)$, let $E(X,Y)=\{ E\in \mh: E\cap X \neq \emptyset, E\cap Y\neq \emptyset \text{~and~} E\subseteq X\cup Y \}.$
 For  $S\subseteq V(\mh)$, let $E_\mh(S)=\{E\in \mh: S\subseteq E\}$. We write $E_\mh(v)$ instead of $E_\mh(\{v\})$. The degree of $v$ in $\mh$, denoted by $d_{\mh}(v)$, is $|E_\mh(v)|$.
Let $u,v\in V(\mh)$. We call  $(u,v)$ a \textbf{good pair}, if there exist  $E_1,E_2\in \mh$ with $E_1\not=E_2$ such that $u\in E_1$ and $v\in E_2$. If  there exists an order of $V(\mh)$, say $\{v_1,\dots,v_{|V(\mh)|}\}$, such that $(v_i,v_{i+1})$ is a good pair for $1\le i\le |V(\mh)|-1$, then we call such order of $V(\mh)$  a \textbf{good order}.
Then we have the following result.
\begin{lemma}\label{lem: good order} Let $\mh$
    be an $r$-uniform hypergraph of order $n$  and size $m\geq 2$. Then  $V(\mh)$ has a good order. Moreover, any vertex of $V(\mh)$ could be the first one in a good order.
\end{lemma}
 \begin{proof}
     We will prove by induction on  $m$. Let $m=2$. Assume  $E_1\setminus E_2=\{u_1,\dots,u_t\}$ and  $E_2\setminus E_1=\{v_1,\dots,v_t\}$ for some $1\leq t\leq r$, and  $E_1\cap E_2=\{w_1,\dots,w_{r-t}\}$. Then $\{u_1,v_1,u_2,v_2,\dots,u_t,v_t,w_1,w_2,\dots,w_{r-t}\}$ is a good order. Since we can rotate the order and the resulted order remains a good order, any vertex could be the  first one.

    Assume $m\geq 3$. Let $E\in \mh$ and denote $\mh'=\mh\setminus\{E\}$. Then $|V(\mh')|\geq r$ and $|\mh'|=m-1$. By induction hypothesis, there exist a good order of $V(\mh')$, say $\{u_1',\dots,u_{|V(\mh')|}'\}$, and $u_1'$ could be any vertex in $\mh'$. Assume $E\setminus V(\mh')=\{v_1',\dots,v_a'\}$, where $0\leq a\leq r$. If $a=0$, then $\{u_1',\dots,u_{|V(\mh')|}'\}$ is also a good order of $V(\mh)$. If $a\geq 1$, then it's easy to check $\{v_1',u_1',\dots,v_a',u_a',u_{a+1}',\dots,u_{|V(\mh')|}'\}$ is a good order. Since we can rotate the order and the resulting order remains good, any vertex could could be the  first one.
 \end{proof}

 \begin{lemma}\label{lemma: first}
     For integers $\ell'\geq r\geq 3$, we have
     \begin{equation*}
         \binom{2\ell'-1}{r-1}>\frac{1}{2\ell'}\binom{2\ell'}{r}+\frac{2}{2\ell'}\binom{2\ell'}{r-1}+\frac{1}{2\ell'}\binom{2\ell'}{r-2}.
     \end{equation*}
 \end{lemma}
 \begin{proof}
Note that $\ell'\geq r$. Thus we have
     \begin{equation*}
    \begin{aligned}
        &\binom{2\ell'-1}{r-1}-\left(\frac{1}{2\ell'}\binom{2\ell'}{r}+\frac{2}{2\ell'}\binom{2\ell'}{r-1}+\frac{1}{2\ell'}\binom{2\ell'}{r-2}\right)\\
        &=\binom{2\ell'-1}{r-1}\left(1-\frac{1}{r}-\frac{2}{2\ell'-r+1}-\frac{r-1}{(2\ell'-r+1)(2\ell'-r+2)}\right)\\
        &\geq \binom{2\ell'-1}{r-1}\left(1-\frac{1}{r}-\frac{2}{r+1}-\frac{r-1}{(r+1)(r+2)}\right)\\
        &\geq \binom{2\ell'-1}{r-1}\left(1-\frac{1}{3}-\frac{2}{4}-\frac{3-1}{(3+1)(3+2)}\right)>0.
    \end{aligned}
\end{equation*}
 \end{proof}
 \begin{lemma}\label{lem: the ineuqlaity in sec 4 (1)}
 For integers $r,k,\ell$ with $k\geq 2$ and $\ell\geq r\geq 3$, we have
     $$\binom{k\ell-1}{r-1}-\frac{1}{2}\binom{k\ell-1}{r-2}\geq \binom{(k-1)\ell}{r-1}+1.$$
 \end{lemma}
 \begin{proof}
    Note that $\binom{k\ell-1}{r-1}-\binom{(k-1)\ell}{r-1}\geq \binom{k\ell-2}{r-2}$. Since
     \begin{equation*}
         \begin{aligned}
             &2\binom{k\ell-2}{r-2}-\binom{k\ell-1}{r-2}-2\\
             \geq & 2~\frac{(k\ell-2)\dots(k\ell-r+1)}{(r-2)!}-\frac{(k\ell-1)\dots(k\ell-r+2)}{(r-2)!}-2\\
             \geq &\frac{(k\ell-2)\dots(k\ell-r+2)}{(r-2)!}\left( k\ell-2r+3-\frac{2(r-2)!}{(k\ell-2)\dots(k\ell-r+2)}\right)>0,
         \end{aligned}
     \end{equation*}the result holds.
     The last inequality holds because $\frac{2(r-2)!}{(k\ell-2)\dots(k\ell-r+2)}\leq 2$ and $k\ell\geq 2r$ when $k\geq 2$ and $\ell\geq r$. 
 \end{proof}
   \begin{lemma}\label{lem: the ineuqlaity in sec 3 (2)}
 For integers $r,k,\ell$ with $k\geq 3$ and $\ell\geq r\geq 3$, we have
     $$\frac{1}{2}\sum_{t=1}^{r-2}\binom{(k-1)\ell-1}{r-t-1}-\ell+\binom{\ell-1}{r-2}>0.$$
 \end{lemma}
 \begin{proof}
     When $r=3$, we have
     \begin{equation*}
         \begin{aligned}
             &\frac{1}{2}\sum_{t=1}^{r-2}\binom{(k-1)\ell-1}{r-t-1}-\ell+\binom{\ell-1}{r-2}=\frac{1}{2}\binom{(k-1)\ell-1}{1}-1\\
             &\geq \frac{1}{2}\binom{2\ell-1}{1}-1
             =\ell-\frac{3}{2}>0.
         \end{aligned}
     \end{equation*}
     When $r\geq 4$ and $\ell=r$, we have
     \begin{equation*}
         \begin{aligned}
             &\frac{1}{2}\sum_{t=1}^{r-2}\binom{(k-1)\ell-1}{r-t-1}-\ell+\binom{\ell-1}{r-2}=\frac{1}{2}\sum_{t=1}^{r-2}\binom{(k-1)\ell-1}{r-t-1}-1\\
             &>\frac{1}{2}\binom{2\ell-1}{1}-1
             =\ell-\frac{3}{2}>0.
         \end{aligned}
     \end{equation*}
     When $r\geq 4$ and $\ell\geq r+1\ge 5$, by the monotonicity of composite numbers, we have
     \begin{equation*}
         \begin{aligned}
             &\frac{1}{2}\sum_{t=1}^{r-2}\binom{(k-1)\ell-1}{r-t-1}-\ell+\binom{\ell-1}{r-2}>\binom{\ell-1}{r-2}-\ell\\
             &\geq \binom{\ell-1}{2}-\ell
             =\frac{\ell^2-5\ell+2}{2}>0.
        \end{aligned}
    \end{equation*}
 \end{proof}
 \begin{lemma}\label{lem: the ineuqlaity in sec 3 (3)}
 For integers $r,k,\ell$ with $k\geq 2$ and $\ell\geq r\geq 3$, we have
 $$\binom{k\ell-1}{r-1}>\binom{(k-1)\ell-1}{r-1}+\binom{k\ell-1}{r-2}.$$
  \end{lemma}
\begin{proof}
 When $k\geq 2$ and $r=3$, we have
     \begin{equation*}
        \begin{aligned}
            &\binom{k\ell-1}{3-1}-\binom{(k-1)\ell-1}{3-1}-\binom{k\ell-1}{3-1}=\frac{1}{2}\left((2k-1)\ell^2-(k+3)\ell+1\right)\\
            &\geq \frac{1}{2}\left(9(2k-1)-3(k+3)+1\right)
            =\frac{1}{2}(15k-17)>0
        \end{aligned}
    \end{equation*}
    and the result holds. Assume $4\leq r\leq \ell$. 
    Then
    \begin{equation*}
        \begin{aligned}
            &\binom{k\ell-1}{r-1}-\binom{(k-1)\ell-1}{r-1}-\binom{k\ell-1}{r-2}
            =\frac{k\ell-2r+2}{r-1}\binom{k\ell-1}{r-2}-\binom{(k-1)\ell-1}{r-1}\\
            &>\frac{k\ell-2r+2}{r-1}\left(\binom{(k-1)\ell-1}{r-2}+\binom{k\ell-1}{r-3}\right)-\binom{(k-1)\ell-1}{r-1}\\
            &=\frac{k\ell-2r+2}{r-1}\left(\binom{(k-1)\ell-1}{r-2}+\binom{k\ell-1}{r-3}\right)
            -\frac{(k-1)\ell-r+1}{r-1}\binom{(k-1)\ell-1}{r-2}\\
            &=\frac{\ell-r+1}{r-1}\binom{(k-1)\ell-1}{r-2}+\frac{k\ell-2r+2}{r-1}\binom{k\ell-1}{r-3}>0.
        \end{aligned}
    \end{equation*}The first inequality holds by induction hypothesis. So we finished the proof.
\end{proof}

 \begin{lemma}\label{last}
When $k\geq 2$, $\ell\geq 5$ and $\ell'\geq  r\geq 3$, we have
$$\max\left\{\binom{k\ell'-1}{r-1}-\frac{1}{2}\binom{k\ell'-1}{r-2}+\frac{1}{2},\frac{1}{\ell}\binom{\ell}{r}+\frac{5}{2}\right\}<\binom{k\ell'-1}{r-1},$$ where $\ell'=\lchuto$.
\end{lemma}
\begin{proof}
Since $\binom{k\ell'-1}{r-2}>1$, we have $\binom{k\ell'-1}{r-1}-\frac{1}{2}\binom{k\ell'-1}{r-2}+\frac{1}{2}<\binom{k\ell'-1}{r-1}$. Note that $\ell\geq 5$. Then
\begin{equation*}
    \begin{aligned}
      &\binom{k\ell'-1}{r-1}-\frac{1}{\ell}\binom{\ell}{r}-\frac{5}{2}\geq  \binom{\ell-1}{r-1}-\frac{1}{\ell}\binom{\ell}{r}-\frac{5}{2}\\
      &=(1-\frac{1}{r})\binom{\ell-1}{r-1}-\frac{5}{2}
      \geq \frac{2}{3}\binom{\ell-1}{2}-\frac{5}{2}>0.
    \end{aligned}
\end{equation*}
\end{proof}

\section{Proof of Theorem \ref{thm: berge linear forest} when $k=2$}

In this section we prove  Theorem \ref{thm: berge linear forest} when $k=2$. When $\ell\geq r+6$ is odd and $n$ is sufficiently large, by Theorem \ref{thm: two berge path}, the result holds. 
 So we just need to consider the case when $\ell$ is even. Then $\ell=2\ell'$. Let $\mathcal{H}$ be a $\ber 2P_\ell$-free $r$-graph of order $n$  and size $\mathrm{ex}_{r}(n,\ber 2P_{\ell})$. By (\ref{3-1}), we have
  \begin{equation}\label{eq-0}
 |\mathcal{H}|\ge |\mathcal{H}_1(n,k,\ell)|=\binom{2\ell'-1}{r-1}(n-2\ell'+1)+\binom{2\ell'-1}{r}+\binom{2\ell'-1}{r-2}.\end{equation} Keep in mind that we always suppose $n$ is large enough. Let $\mathcal{H}'$ be the $r$-graph  obtained from $\mathcal{H}$ by repeatedly removing vertices of degree at most $\binom{2\ell'-1}{r-1}-1$ and all hyperedges incident to them. By $\mathcal{H}$ being  $\ber 2P_\ell$-free, the length of any $\ber$path in $\mh'$ is at most $4\ell'$. Let  the longest $\ber$path in $\mh'$ be a $\ber P_{\ell_0}$ and  denote the  component contained the $\ber P_{\ell_0}$ as $\mh_{1}'$. Then $\ell_{0}\leq 2\ell=4\ell'$. Let   $\mh_{2}'=\mh[V(\mh')\setminus V(\mh_{1}')]$.
 \begin{lemma}\label{2-01} We have $\ell_{0}\ge \ell$.
 \end{lemma}
 \begin{proof} Suppose $\ell_0<\ell$. Then $\mh'$ is $\ber P_{\ell}$-free. By Theorem \ref{thm: turan of BPl} (1) and considering the deleted hyperedges, we have
    \begin{align*}
    |\mathcal{H}|&\le \frac{|V(\mh')|}{2\ell'}{2\ell'\choose r}+\left(\binom{2\ell'-1}{r-1}-1\right)(n-|V(\mh')|)\\
    &\leq \max\left\{\frac{1}{2\ell'}{2\ell'\choose r},\left(\binom{2\ell'-1}{r-1}-1\right)\right\}n\\
    &=\left(\binom{2\ell'-1}{r-1}-1\right)n,
    \end{align*}a contradiction with (\ref{eq-0}) when $n$ is large enough.
 \end{proof}
\begin{lemma}\label{2-11} $|V(\mh_{1}')|> \max\{N_{4\ell'+1,r},N_{4\ell',r},N_{4\ell'-1,r}\}$, where $N_{4\ell'+1,r},N_{4\ell',r}$ and $N_{4\ell'-1,r}$ are described as Theorem \ref{thm: connect Pl}.
 \end{lemma}
\begin{proof} Let $s=\max\{N_{4\ell'+1,r},N_{4\ell',r},N_{4\ell'-1,r}\}$.
 By Lemma \ref{2-01} and  $\mathcal{H}$ being  $\ber 2P_\ell$-free, $\mh_{2}'$ is $\ber P_{\ell}$-free. By Theorems \ref{thm: turan of BPl} (1), we have
     \begin{equation*}
        \begin{aligned}
            |\mh_{2}'|&\leq \mathrm{ex}_r(|V(\mh_{2}')|,\ber P_{\ell})=\mathrm{ex}_r(|V(\mh_{2}')|,\ber P_{2\ell'})\\
           &\leq \frac{|V(\mh_{2}')|}{2\ell'}\binom{2\ell'}{r}.
        \end{aligned}
    \end{equation*}Suppose $|V(\mh_{1}')|\leq s$. Then $|\mh_{1}'|\leq \binom{s}{r}$. So for $n$ large enough, we have
    \begin{equation*}
        \begin{aligned}
            |\mh|&\leq \left(\binom{2\ell'-1}{r-1}-1\right)|V(\mh)\setminus V(\mh')|+\binom{s}{r}+\frac{|V(\mh_{2}')|}{2\ell'}\binom{2\ell'}{r} \\
            &=\left(\binom{2\ell'-1}{r-1}-1\right)(n-|V(\mh_{1}')|-|V(\mh_{2}')|)+\binom{s}{r}+\frac{|V(\mh_{2}')|}{2\ell'}\binom{2\ell'}{r}\\
            &\leq \left(\binom{2\ell'-1}{r-1}-1\right)(n-2\ell'+1)+O(1)<\binom{2\ell'-1}{r-1}n+O(1),
        \end{aligned}
    \end{equation*}a contradiction with (\ref{eq-0}).
 \end{proof}
\begin{lemma}\label{lem: contains P4l}
    We have $4\ell'-1\leq\ell_0\leq 4\ell'$.
\end{lemma}
\begin{proof}
    Assume $\ell_0\leq 4\ell'-2$. By Lemma \ref{2-11}, Theorem \ref{thm: connect Pl} and $\ell\geq r+7$, we have\begin{equation*}
        \begin{aligned}
            |\mh_{1}'|&\leq \mathrm{ex}_r^{con}(|V(\mh_{1}')|,\ber P_{\ell_0+1})\leq \mathrm{ex}_r^{con}(|V(\mh_{1}')|,\ber P_{4\ell'-1})\\
            &= \binom{2\ell'-1}{r-1}(|V(\mh_{1}')|-2\ell'+1)+\binom{2\ell'-1}{r}.
        \end{aligned}
    \end{equation*}Since $\mh_{2}'$ is $\ber P_{\ell}$-free, by Theorems \ref{thm: turan of BPl} (1), we have
     $|\mh_{2}'|\leq  \frac{|V(\mh_{2}')|}{2\ell'}\binom{2\ell'}{r}.$
   So we have
    \begin{equation*}
        \begin{aligned}
            |\mh|&\leq \left(\binom{2\ell'-1}{r-1}-1\right)|V(\mh)\setminus V(\mh')|+\binom{2\ell'-1}{r-1}|V(\mh_{1}')|+\frac{|V(\mh_{2}')|}{2\ell'}\binom{2\ell'}{r} \\
            &+\binom{2\ell'-1}{r}-(2\ell'-1)\binom{2\ell'-1}{r-1}\\
            &\leq \binom{2\ell'-1}{r-1}(n-2\ell'+1)+\binom{2\ell'-1}{r}\\
            &< \binom{2\ell'-1}{r-1}(n-2\ell'+1)+\binom{2\ell'-1}{r}+\binom{2\ell'-1}{r-2},
        \end{aligned}
    \end{equation*}
   a contradiction with (\ref{eq-0}).
\end{proof}

Now we are going to prove  Theorem \ref{thm: berge linear forest} when $k=2$.

\noindent\emph{Proof of Theorem \ref{thm: berge linear forest} when $k=2$ } Let $\mathcal{H}$ be a $\ber 2P_\ell$-free $r$-graph of order $n$  and size $\mathrm{ex}_{r}(n,\ber 2P_{\ell})$, where $\ell$ is even and $\ell=2\ell^{'}$. Let $\mathcal{H}'$ be the $r$-graph  obtained from $\mathcal{H}$ by repeatedly removing vertices of degree at most $\binom{2\ell'-1}{r-1}-1$ and all hyperedges incident to them.
Let $\mathcal{P}_{\ell_0}$ be a longest $\ber$path in $\mh'$ and  denote the  component contained $\mathcal{P}_{\ell_0}$ as $\mh_{1}'$. By Lemma \ref{lem: contains P4l}, $4\ell'-1\leq\ell_0\leq 4\ell'$. We complete the proof by considering two cases.

 \noindent{\bf Case 1.} $\ell_0= 4\ell'$.

Let $U=\{u_1,\dots,u_{4\ell'+1}\}$ and $\mathcal{F}=\{F_1,\dots,F_{4\ell'}\}$ be the defining vertices  and the defining hyperedges of $\mathcal{P}_{4\ell'}$ in $\mh'$, respectively. Since $\mathcal{H}$ is $\ber 2P_\ell$-free, we have
      $N_{\mh'\setminus \mathcal{ F}}(u_1)\cup N_{\mh'\setminus \mathcal{ F}}(u_{4\ell'+1})\subseteq U$.  By  Lemma \ref{2-11} and Theorem \ref{thm: connect Pl}, we have
    \begin{equation*}
        \begin{aligned}
            |\mh_{1}'|&\leq \mathrm{ex}_{r}^{con}(|V(\mh_{1}')|,\ber P_{4\ell'+1})
            =\binom{2\ell'}{r-1}|V(\mh_{1}')|+O(1).
        \end{aligned}
    \end{equation*}Thus we have
\begin{equation*}
        \begin{aligned}
            |\mh|&\leq \left(\binom{2\ell'-1}{r-1}-1\right)|V(\mh)\setminus V(\mh')|+\binom{2\ell'}{r-1}|V(\mh_{1}')|+\frac{|V(\mh_{2}')|}{2\ell'}\binom{2\ell'}{r} +O(1)\\
            &\leq \binom{2\ell'}{r-1}|V(\mh_{1}')|+\max\left\{\binom{2\ell'-1}{r-1}-1,\frac{1}{2\ell'}\binom{2\ell'}{r}\right\}(n-|V(\mh_{1}')|)+O(1)\\
            &= \binom{2\ell'}{r-1}|V(\mh_{1}')|+\left(\binom{2\ell'-1}{r-1}-1\right)(n-|V(\mh_{1}')|)+O(1).
        \end{aligned}
    \end{equation*}
   By (\ref{eq-0}), we have
    $$|V(\mh_{1}')|\geq \left(\binom{2\ell'}{r-1}-\binom{2\ell'-1}{r-1}+1\right)^{-1}(n+O(1))>0,$$
   which implies $|V(\mh_{1}')|$ is large enough.
We  have the following result.
\begin{lemma}\label{claim: inter degree of ends vertices}
      $|N_{\mh'\setminus\mathcal{F}}(u_1)|\geq 2\ell'-1$ and $|N_{\mh'\setminus \mathcal{F}}(u_{4\ell'+1})|\geq 2\ell'-1$.
\end{lemma}
\begin{proof} Note that
      $N_{\mh'\setminus \mathcal{ F}}(u_1)\cup N_{\mh'\setminus \mathcal{ F}}(u_{4\ell'+1})\subseteq U$. 
       We  assume $\mathcal{P}$ is a $\ber P_{4\ell'}$  minimizing the sum $d_{\mathcal{F}}(u_1)+d_{\mathcal{F}}(u_{4\ell'+1})$. If $u_1\in F_i$ for some $2\le i\le 4\ell'+1$, then $$\{u_i,F_{i-1},u_{i-1},F_{i-2},\dots,u_2,F_1,u_1,F_i,u_{i+1},F_{i+1},\dots,u_{4\ell'},F_{4\ell'},u_{4\ell'+1}\}$$
    is a $\ber P_{4\ell'}$. By the minimality of $d_{\mathcal{F}}(u_1)+d_{\mathcal{F}}(u_{4\ell'+1})$, we have $d_{\mathcal{F}}(u_1)\leq d_{\mathcal{F}}(u_i)$. Denote $X=\{(u,F):F\in \mathcal{F},u\in F\cap U\}$. Then $d_{\mathcal{F}}(u_1)^2\le |X|\le r|\mathcal{F}|=4r\ell'$ which implies
    $d_{\mathcal{F}}(u_1)\leq\sqrt{4r\ell'}$. By the same argument, we have  $d_{\mathcal{F}}(u_{4\ell'+1})\leq \sqrt{4r\ell'}$. Note that $d_{\mh'}(u_1)\ge \binom{2\ell'-1}{r-1}$.
    When $2\ell'\geq r+6$, we have $d_{\mh'\setminus \mathcal{F}}(u_1)\ge \binom{2\ell'-1}{r-1}-\sqrt{4r\ell'}> \binom{2\ell'-2}{r-1}$.  Thus $|N_{\mh'\setminus \mathcal{F}}(u_1)|\geq 2\ell'-1$ and similarly, $|N_{\mh'\setminus \mathcal{F}}(u_{4\ell'+1})|\geq 2\ell'-1$.
\end{proof}

\vskip.2cm
For a set $S\subseteq U$, define $S^{-}=\{u_{i}:~u_{i+1}\in S\}$ and $S^{+}=\{u_{i}:~u_{i-1}\in S\}$. Define $S_{1}=N_{\mh_{1}'\setminus \mathcal{F}}(u_1)$ and $S_{4\ell'+1}=N_{\mh_{1}'\setminus \mathcal{F}}(u_{4\ell'+1})$ for short. Then $S_1\cup S_{4\ell'+1}\subseteq U$ and we have the following lemma.
\begin{lemma}\label{le-2}
     $\mh_{1}'$ contains a Berge cycle of length $4\ell'$.
\end{lemma}
\begin{proof} By contradiction.
We first have the following facts.
\par\noindent
\textbf{Fact 1:} If  $u_i\in S_1$ (resp. $u_i\in S_{4\ell'+1}$), then $N_{\mh_{1}'\setminus \mathcal{F}}(u_{i-1})\subseteq U$ (resp. $N_{\mh_{1}'\setminus \mathcal{F}}(u_{i+1})\subseteq U$), where $2\le i\le 4\ell'+1$ (resp. $1\le i\le 4\ell'$).

\noindent \textbf{Proof of Fact 1:} Suppose there is $2\le i\le 4\ell'+1$ such that $u_i\in S_1$ and $N_{\mh_{1}'\setminus \mathcal{F}}(u_{i-1})\not\subseteq U$. Then there are $F_{0},\hat{F}_{0}\in \mh_{1}'\setminus\mathcal{F}$ such that $\{u_{1},u_{i}\}\subseteq F_{0}$ and $\{u_0,u_{i-1}\}\subseteq \hat{F}_{0}$, where $u_{0}\notin U$. Then  $\{u_{4\ell'+1},F_{4\ell'},\ldots,F_{i},u_{i},F_{0},u_1,F_{1},\ldots,F_{i-2},u_{i-1},\hat{F}_{0},u_0\}$ is a $\ber P_{4\ell'+1}$ in $\mathcal{H}'$, a contradiction with Lemma \ref{lem: contains P4l}. \q 
\par\noindent
\textbf{Fact 2:} $S_{1}^{-}\cap S_{4\ell'+1}^{+}=\emptyset.$

\noindent \textbf{Proof of Fact 2:} Suppose there is $u_i\in S_{1}^{-}\cap S_{4\ell'+1}^{+}$, where $2\le i\le 4\ell'$.
Then there exist $F_{0},\hat{F}_{0}\in \mh_{1}'\setminus\mathcal{F}$ such that $\{u_1,u_{i+1}\}\subseteq F_0,\{u_{i-1},u_{4\ell'+1}\}\subseteq\hat{F}_{0}$. If $F_0=\hat{F}_0$, $\{u_{1},u_{4\ell'+1}\}\subseteq F_{0}$ and thus leads to a $\ber C_{4\ell'+1}$. Since $|V(\mh_{1}')|$ is large enough, there exists a vertex $x_0\in V(\mh_{1}')$ such that $x_0\notin V(\ber C_{4\ell'+1})$. Since $\mh_{1}'$ is connected, there is a $\ber P_{4\ell'+1}$ in $\mh_{1}'$, a contradiction  with Lemma \ref{lem: contains P4l}. Thus $F_0\not=\hat{F}_0$.
Therefore $$\{u_1,F_{0},u_{i+1},F_{i+1},\ldots,F_{4\ell'},u_{4\ell'+1},\hat{F}_{0},u_{i-1},F_{i-2},\ldots,F_{1},u_1\}$$ is a $\ber C_{4\ell'}$ in $\mh_{1}'$, a contradiction.\q

By Lemma \ref{claim: inter degree of ends vertices} and Fact 2, we easily have the following fact.

\par\noindent
\textbf{Fact 3:}  $|S_{1}^{-}|=|S_1|\geq 2\ell'-1$, $|S_{4\ell'+1}^{+}|=|S_{4\ell'+1}|\geq 2\ell'-1$ and $|S_{1}^{-} \cup S_{4\ell'+1}^{+}|\geq 4\ell'-2$.
\par
\vskip.2cm
Now we complete the proof of Lemma \ref{le-2}. By Fact 3, we have $|U\setminus (S_{1}^{-}\cup S_{4\ell'+1}^{+})|\le 3$.
\par
Since $|V(\mh_{1}')|$ is large enough, there exists $u\in V(\mh_{1}')\setminus (\cup_{i=1}^{4\ell'}F_i)$. If $U=S_{1}^{-}\cup S_{4\ell'+1}^{+}$, by Fact 1, there is not a  $E\in \mh_{1}'\setminus\mathcal{F}$ such that $u\in E$ and $E\cap U\not=\emptyset$. By $\mh_{1}'$ being connected, there  exists $v\in F_i\setminus U$ and $F\notin \mathcal{F}$ such that $\{u,v\}\subseteq F$, where $1\le i\le 4\ell'$. Recall that $u_{i}\in S_{1}^{-}\cup S_{4\ell'+1}^{+}$, say $u_{i}\in S_{1}^{-}$. Then $u_{i+1}\in S_{1}$ which implies there exists $\hat{F}\in \mh_{1}'\setminus \mathcal{F}$ such that $\{u_1,u_{i+1}\}\subseteq \hat{F}$. By Fact 1,  $\hat{F}\not=F$. Thus $\{u_{4\ell'+1},F_{4\ell'},\ldots,F_{i+1},u_{i+1},\hat{F},u_1,F_1,\ldots,u_{i},F_i,v,F,u\}$ leads a $\ber P_{4\ell'+2}$, a contradiction  with Lemma \ref{lem: contains P4l}. Hence $|U\setminus (S_{1}^{-}\cup S_{4\ell'+1}^{+})|\ge 1$.
\par
Denote $A=U\setminus (S_{1}^{-}\cup S_{4\ell'+1}^{+})$ for short. Let $A=\{u_{j_1},u_{j_2},u_{j_3}\}$ (resp. $A=\{u_{j_2},u_{j_3}\}$ and $A=\{u_{j_3}\}$) when $|A|=3$ (resp. $|A|=2$ and $|A|=1$), where $j_1<j_2<j_3$. Let $\mh_{1}'=\mathcal{F}\cup\hat{\mathcal{F}}\cup \overline{\mathcal{F}}$, where
\begin{equation*}
     \hat{\mathcal{F}}=\{E\in \mh_{1}'\setminus\mathcal{F} : E\subseteq U\} \text{, and~}
    \overline{\mathcal{F}}=\{E\in \mh_{1}'\setminus\mathcal{F} : E\setminus U\neq \emptyset\}.
\end{equation*}
When $|A|=3$ (resp. $|A|=2$), we have either $j_1\not=2\ell'+1$ (resp. $j_2\not=2\ell'+1$) or $j_3\not=2\ell'+1$.  We will assume $j_3\neq 2\ell'+1$. So  $\{u_1,F_{1},\ldots,F_{j_{3}-2},u_{j_3-1}\}$ or $\{u_{j_3+1},F_{j_3+1},\ldots,F_{4\ell'},u_{4\ell'+1}\}$ contains a $\ber P_{2\ell'}$  for $2\le |A|\le 3$.
Note that for any $E\in \overline{\mathcal{F}}$, $E\cap U\subseteq A$. 
For $|A|=3$, let
\begin{equation*}
    \begin{aligned}
        &\mathcal{F}_{3,1}=\{E\in \overline{\mathcal{F}} : E\cap \{u_{j_1},u_{j_2}\}=\{u_{j_1}\}\}
        ,~\mathcal{F}_{3,2}=\{E\in \overline{\mathcal{F}} : E\cap \{u_{j_1},u_{j_2}\} =\{u_{j_2}\}\},\\
        &\mathcal{F}_{3,3}=\{E\in \overline{\mathcal{F}} : E\cap \{u_{j_1},u_{j_2}\}=\{u_{j_1},u_{j_2}\}\},~\mathcal{F}_{3,4}=\{E\in \overline{\mathcal{F}} :  \text{$E\cap A=\{u_{j_3}\}$ or $E\cap U=\emptyset$}\}.
    \end{aligned}
\end{equation*}
For $|A|=2$, let
\begin{equation*}
    \begin{aligned}
        &\mathcal{F}_{2,1}=\{E\in \overline{\mathcal{F}} : E\cap \{u_{j_2},u_{j_3}\}=\{u_{j_2}\}\}
        ,~\mathcal{F}_{2,2}=\{E\in \overline{\mathcal{F}} : E\cap \{u_{j_2},u_{j_3}\} =\{u_{j_3}\}\},\\
        &\mathcal{F}_{2,3}=\{E\in \overline{\mathcal{F}} : E\cap \{u_{j_2},u_{j_3}\}=\{u_{j_2},u_{j_3}\}\},\text{and}~\mathcal{F}_{2,4}=\{E\in \overline{\mathcal{F}} :  \text{ $E\cap U=\emptyset$}\}.
    \end{aligned}
\end{equation*}For $|A|=1$, let
\begin{equation*}
        \mathcal{F}_{1,1}=\{E\in \overline{\mathcal{F}} : E\cap \{u_{j_3}\} =\{u_{j_3}\}\},~\mathcal{F}_{1,2}=\{E\in \overline{\mathcal{F}} :  \text{  $E\cap U=\emptyset$}\}.
\end{equation*}
Recall that $\mh_{1}'$ is $\ber 2P_{2\ell'}$-free and $\{u_1,F_{1},\ldots,F_{j_{3}-2},u_{j_3-1}\}$ or $\{u_{j_3+1},F_{j_3+1},\ldots,F_{4\ell'},u_{4\ell'+1}\}$ is  a $\ber P_{2\ell'}$. Set $\mathcal{F}_{3,i}'=\{E\setminus\{u_{j_i}\}:~E\in \mathcal{F}_{3,i}\}$, $i=1,2$. Then  $|\mathcal{F}_{3,i}|=|\mathcal{F}_{3,i}'|$ and $\mathcal{F}_{3,i}'$ is an $(r-1)$-uniform $\ber P_{2\ell'}$-free hypergraph for $i=1,2$. Set $\mathcal{F}_{3,3}'=\{E\setminus\{u_{j_1},u_{j_2}\}:~E\in \mathcal{F}_{3,3}\}$. When $r\geq 4$, $\mathcal{F}_{3,3}'$ is an $(r-2)$-uniform $\ber P_{2\ell'}$-free hypergraph and we have $|\mathcal{F}_{3,3}|=|\mathcal{F}_{3,3}'|$. When $r=3$, $|\mathcal{F}_{3,3}|\leq |V(\mh_{1}')|$. Notice that $\mathcal{F}_{3,4}$ is an $r$-uniform $\ber P_{2\ell'}$-free hypergraph. By Theorem \ref{thm: turan of BPl} (1), we have  the following upper bounds.
\begin{equation*}
    \begin{aligned}
        &|\mathcal{F}_{3,1}|\leq \frac{1}{2\ell'}\binom{2\ell'}{r-1}|V(\mh_{1}')|, ~|\mathcal{F}_{3,2}|\leq \frac{1}{2\ell'}\binom{2\ell'}{r-1}|V(\mh_{1}')|,\\
        &|\mathcal{F}_{3,3}|\leq \frac{1}{2\ell'}\binom{2\ell'}{r-2}|V(\mh_{1}')|,~\text{and}~|\mathcal{F}_{3,4}|\leq \frac{1}{2\ell'}\binom{2\ell'}{r}|V(\mh_{1}')|.\\
    \end{aligned}
\end{equation*}By the same argument for $|A|=2$ and  $|A|=1$, we have
\begin{align*}
        &|\mathcal{F}_{2,1}|\leq \frac{1}{2\ell'}\binom{2\ell'}{r-1}|V(\mh_{1}')|, ~|\mathcal{F}_{2,2}|\leq \frac{1}{2\ell'}\binom{2\ell'}{r-1}|V(\mh_{1}')|,\\
        &|\mathcal{F}_{2,3}|\leq \frac{1}{2\ell'}\binom{2\ell'}{r-2}|V(\mh_{1}')|,~|\mathcal{F}_{2,4}|\leq \frac{1}{2\ell'}\binom{2\ell'}{r}|V(\mh_{1}')|,\\
        &|\mathcal{F}_{1,1}|\leq \frac{1}{2\ell'}\binom{2\ell'}{r-1}|V(\mh_{1}')|,~|\mathcal{F}_{1,2}|\leq \frac{1}{2\ell'}\binom{2\ell'}{r}|V(\mh_{1}')|.
    \end{align*}So we just need to consider the case $|A|=3$.
When $|A|=3$, we have
\begin{align*}
        |\mh|&\leq |\mh'|+\left(\binom{2\ell'-1}{r-1}-1\right)|V(\mh)\setminus V(\mh')|\\
        &\leq |\mh_{1}'|+|\mh_{2}'|+\left(\binom{2\ell'-1}{r-1}-1\right)|V(\mh)\setminus V(\mh')|\\
        &\leq |\mathcal{F}|+|\hat{\mathcal{F}}|+\sum_{i=1}^{4}|\mathcal{F}_{3,i}|+\mathrm{ex}_{r}(|V(\mh_{2}')|,\ber P_{2\ell'})+\left(\binom{2\ell'-1}{r-1}-1\right)|V(\mh)\setminus V(\mh')|\\
        &\leq  \left(\frac{1}{2\ell'}\binom{2\ell'}{r}+\frac{2}{2\ell'}\binom{2\ell'}{r-1}+\frac{1}{2\ell'}\binom{2\ell'}{r-2}\right)|V(\mh_{1}')|+4\ell'+\binom{4\ell'+1}{r}\\
        &+\frac{1}{2\ell'}\binom{2\ell'}{r}|V(\mh_{2}')|+\left(\binom{2\ell'-1}{r-1}-1\right)(n-|V(\mh_{1}')|-|V(\mh_{2}')|)\\
        &< \binom{2\ell'-1}{r-1}|V(\mh_{1}')|+\binom{2\ell'-1}{r}+\binom{2\ell'-1}{r-2}-(2\ell'-1)\binom{2\ell'-1}{r-1}\\
                &+\frac{1}{2\ell'}\binom{2\ell'}{r}|V(\mh_{2}')|+\left(\binom{2\ell'-1}{r-1}-1\right)(n-|V(\mh_{1}')|-|V(\mh_{2}')|)\\
        &\leq \binom{2\ell'-1}{r-1}(n-2\ell'+1)+\binom{2\ell'-1}{r}+\binom{2\ell'-1}{r-2},
    \end{align*}a contradiction with (\ref{eq-0}).
 Note that the fifth inequality comes from Lemma \ref{lemma: first} and when $n$ is large enough.
\end{proof}

Let $S_{\ell,v}$ be a star with $\ell$ edges and center $v$. Khormali and Palmer \cite{KHORMALI2022103506} proved the following lemma establishing the degree conditions for the existence of a $\ber S_{\ell,x}$.
\begin{lemma}[ \cite{KHORMALI2022103506}]\label{lem: berge star}
Fix integers $\ell>r\ge 2$ and let $\mh$ be an $r$-graph. If $x$ is a vertex of degree $d_{\mh}(x)>\binom{\ell-1}{r-1}$, then $\mh$ contains a $\ber S_{\ell,x}$ with center $x$.
\end{lemma}

By Lemma \ref{le-2}, we denoted $\mathcal{C}_{4\ell'}$ as the $\ber C_{4\ell'}$ cycle
in $\mh_{1}'$. The defining vertices set and the defining hyperedges set of $\mathcal{C}_{4\ell'}$ are $U_{C}=\{u_{C,1},\ldots,u_{C,4\ell'}\}$ and $\mathcal{F}_{C}=\{F_{C,1},\ldots,F_{C,4\ell'}\}$, respectively. Then all  hyperedges in $\mh_{1}'$ will intersect with $U_{C}$; otherwise together with $\mathcal{F}_{C}$ and the connectedness of $\mh_{1}'$, we can find a $\ber P_{4\ell'+1}$, a contradiction with Lemma \ref{lem: contains P4l}. Now we construct a new hypergraph $\hat{\mh}_{1}=E(\mh_{1}')\setminus \mathcal{F}_{C}$. Define $\mathcal{F}_{i}=\{E\in \hat{\mh}_{1} : |E\cap U_{C}|=i\}$, $1\le i\le r$. Then $\hat{\mh}_{1}=\cup_{i=1}^{r}\mathcal{F}_{i}$. For $1\le j\le r-1$, denote $V_j=(\cup_{E\in \mathcal{F}_{j}}E)\setminus U_C$.

Suppose there are $E_1,E_2\in \cup_{i=1}^{r-2}\mathcal{F}_{i}$ such that $(E_{1}\setminus U_C)\cap (E_{2}\setminus U_C)\not=\emptyset$, say $u_0\in (E_{1}\setminus U_C)\cap (E_{2}\setminus U_C) $.
Since $|E_{i}\setminus U_{C}|\geq 2$ for $i=1,2$, there exists  $u_{0}'\in E_2\setminus (U_{C}\cup\{u_0\})$. Let $u_{j}\in E_1\cap U_{C}$ for some $1\le j\le 4\ell'$. Then
$$\{u_{C,j+1},F_{C,j+1},\ldots,u_{C,4\ell'},F_{C,4\ell'},u_{C,1},F_{C,1},\ldots,F_{C,j-1},u_{C,j},E_1,u_{0},E_{2},u_{0}'\}$$ is  a $\ber P_{4\ell'+1}$ in $\mh_{1}'$, a contradiction  with Lemma \ref{lem: contains P4l}. Hence for any $E_1,E_2\in \cup_{i=1}^{r-2}\mathcal{F}_{i}$, we have $(E_{1}\setminus U_C)\cap (E_{2}\setminus U_C)=\emptyset$. Similarly, for any $E_1\in \mathcal{F}_{r-1} $ and $E_2\in \cup_{i=1}^{r-2}\mathcal{F}_{i} $, we have $(E_{1}\setminus U_C)\cap (E_{2}\setminus U_C)=\emptyset$. Thus  $V_{j_1}\cap V_{j_2}=\emptyset$ for  $1\le j_1\neq j_2\le r-1$, which implies $$\sum_{j=1}^{r-1}|V_j|\leq |V(\hat{\mh}_1)|\leq  \sum_{j=1}^{r-1}|V_{j}|+|U_{C}| \,\text{and}\,
|\mathcal{F}_{i}|=\frac{|V_{i}|}{r-i}, i=1,2,\ldots,r-2.$$
So we have
\begin{equation}\label{equ: (1)}
    \begin{aligned}
        |\hat{\mh}_{1}|&=\sum_{i=1}^{r}|\mathcal{F}_{i}|=\sum_{j=1}^{r-2}\frac{|V_{j}|}{r-j}+|\mathcal{F}_{r-1}|+|\mathcal{F}_{r}|\\
        &\leq \sum_{j=1}^{r-2}\frac{|V_{j}|}{r-j}+|\mathcal{F}_{r-1}|+\binom{4\ell'}{r}=\sum_{j=1}^{r-2}\frac{|V_{j}|}{r-j}+\sum_{v\in V_{r-1}}d_{\hat{\mh}_{1}}(v)+\binom{4\ell'}{r}.
    \end{aligned}
\end{equation}Let $V_{r-1}=V_{\geq}\cup V_{<}$, where
\begin{equation*}
    \begin{aligned}
        &V_{\geq }=\left\{v\in V_{r-1} : d_{\hat{\mh}_{1}}(v)\geq \binom{2\ell'-2}{r-1}+1\right\},\\
        &V_{< }=\left\{v\in V_{r-1} : d_{\hat{\mh}_{1}}(v)\leq \binom{2\ell'-2}{r-1}\right\}.
    \end{aligned}
\end{equation*}
For each vertex $v\in V_{\geq}$, by Lemma \ref{lem: berge star}, there is a $\ber S_{2\ell'-1,v}$ with center $v$ (also the defining vertex) in $\hat{\mh}_{1}$. Denote  $V(\ber S_{2\ell'-1,v})=U_{C,v}\cup \{v\}$, where $U_{C,v}\subseteq U_{C}$.

\vskip.2cm
\begin{claim}\label{claim: V is finite}
 $|V_{\geq}|\leq n_{\ell,r}$, where $n_{\ell,r}=\binom{4\ell'}{2\ell'-1}$.
\end{claim}

\noindent{\bf Proof of Claim \ref{claim: V is finite}} Suppose  $|V_{\geq}|\geq n_{\ell,r}+1$. Then there exist $v_{1},v_{2}\in V_{\geq}$ such that $U_{C,v_1}=U_{C,v_2}$. If there are two vertices of $U_{C,v_1}$ whose indexes are sequential, say $u_{C,1},u_{C,4\ell'}\in U_{C,v_1}$, then we can choose $\{u_{C,1},v_1\}\subseteq E_{1}\in E(\ber S_{2\ell'-1,v_1})$ and $\{u_{C,4\ell'},v_2\}\subseteq E_{2}\in E(\ber S_{2\ell'-1,v_2})$ such that $\{v_1,E_1,u_{C,1},F_{C,1},u_{C,2}, \ldots,F_{C,2\ell'-1},u_{C,2\ell'}\}$ and $$\{v_2,E_2,u_{C,4\ell'},F_{C,4\ell'-1},u_{C,4\ell'-1},\ldots,F_{C,2\ell'+1}, u_{C,2\ell'+1}\}$$ forms a $\ber 2P_{2\ell'}$, a contradiction.

So we may assume that $U_{C,v_1}=\{u_{C,j_1},u_{C,j_2},\ldots,u_{C,j_{2\ell'-1}}\}$, where $j_1<j_2<\cdots<j_{2\ell'-1}$ with $j_{p+1}-j_{p}\ge 2$ for $1\le p\le 2\ell'-2$. Since $|U_C|=4\ell'$, we have $j_{p+1}-j_{p}\le 4$ for $1\le p\le 2\ell'-2$. We will consider two cases.
\par\noindent
\textbf{Case a: } There is $p$, say $p=1$, such that $j_2-j_1=4$.

In this case,  $ j_{p+1}-j_{p}=2$ for $2\le p\le 2\ell'-2$. Assume $j_1=1$. Then $U_{C,v_1}=\{u_{C,1},u_{C,5},\ldots,u_{C,4\ell'-3}\}$. Note that $2\ell'\ge r+7\ge10$. We have $u_{C,2\ell'-1}\in U_{C,v_1}$.
Choose  $\{u_{C,1},v_1\}\subseteq E_{1}\in E(\ber S_{2\ell'-1,v_1})$ and $\{u_{C,2\ell'-1},v_2\}\subseteq E_{2}\in E(\ber S_{2\ell'-1,v_2})$. Thus $\{v_1,E_1,u_{C,1},F_{C,1},$ $u_{C,2}, \ldots,F_{C,2\ell'-1},u_{C,2\ell'-1},E_2,v_2\}$ and $\{u_{C,2\ell'},F_{C,2\ell'},u_{C,2\ell'+1},\ldots,F_{C,4\ell'-1},u_{C,4\ell'}\}$ is a $\ber 2P_{2\ell'}$, a contradiction.
\par\noindent
\textbf{Case b: }$2\le j_{p+1}-j_{p}\le 3$ for $1\le p\le 2\ell'-2$.
\par
In this case, there are exactly $1\le s,t\le 2\ell'-2$ such that $j_{i+1}-j_{i}=3$ for $i\in\{s,t\}$ and $j_{i+1}-j_{i}=2$ for any $i\in\{1,\ldots,2\ell'-2\}\setminus\{s,t\}$.
Assume without loss of generality that $s=1$ and $j_1=1$. Then
 $U_{C,v_1}=\{u_{C,1},u_{C,4},u_{C,6},\ldots,u_{C,2t},u_{C,2t+3},u_{C,2t+5},\ldots,u_{C,4\ell'-1}\}$ and $2\leq t\leq \ell'+2$ by the symmetry of $\ber C_{4\ell'}$.
\par
If $t\leq \ell'-1$, then $u_{C,2\ell'+1}\in U_{C,v_1}$. So
there are $\{u_{C,1},v_1\}\subseteq E_1\in E(\ber S_{2\ell'-1,v_1})$ and $\{u_{C,2\ell'+1},v_2\}\subseteq E_2\in E(\ber S_{2\ell'-1,v_2})$. Thus $\{v_1,E_1,u_{C,1},F_{C,1},\cdots,F_{C,2\ell'-1},u_{C,2\ell'}\}$ and
$\{v_2,E_2,u_{C,2\ell'+1},F_{C,2\ell'+1},\cdots,F_{C,4\ell'-1},u_{C,4\ell'}\}$ form a $\ber 2P_{2\ell'}$, a contradiction.
\par
If $t\geq \ell'+1$, then $u_{C,2\ell'+2}\in U_{C,v_1}$. So
there are
 $\{u_{C,4},v_1\}\subseteq E_1\in E(\ber S_{2\ell'-1,v_1})$ and
$\{u_{C,2\ell'+2},v_2\}\subseteq E_2\in E(\ber S_{2\ell'-1,v_2})$. Thus $\{v_1,E_1,u_{C,4},F_{C,4},\ldots,F_{C,2\ell'+1},u_{C,2\ell'+2},E_2,v_2\}$ and
$\{u_{C,2\ell'+3},F_{C,2\ell'+3},u_{C,2\ell'+4},F_{C,2\ell'+4},\ldots,F_{C,4\ell'-1},u_{C,4\ell'},F_{C,4\ell'},u_{C,1},F_{C,1},\ldots,F_{C,2},u_{C,3}\}$
form a $\ber 2P_{2\ell'}$, a contradiction.
\par
Assume $t=\ell'$. Then $U_{C,v_1}=\{u_{C,1},u_{C,4},u_{C,6},\ldots,u_{C,2\ell'},u_{C,2\ell'+3},u_{C,2\ell'+5},\ldots,u_{C,4\ell'-1}\}.$ Choose  $\{u_{C,1},v_1\}\subseteq E_{1}\in E(\ber S_{2\ell'-1,v_1})$ and $\{u_{C,2\ell'+3},v_2\}\subseteq E_{2}\in E(\ber S_{2\ell'-1,v_2})$. Then $\{v_1,E_1,u_{C,1},F_{C,4\ell'},u_{C,4\ell'}, \ldots,F_{C,2\ell'+3},u_{C,2\ell'+3},E_2,v_2\}$ and $\{u_{C,2},F_{C,2},u_{C,3},\ldots,F_{C,2\ell'+1}, u_{C,2\ell'+2}\}$ forms a $\ber 2P_{2\ell'}$, a contradiction.\q
\par
By Claim \ref{claim: V is finite}, we have
\begin{equation}\label{equ: 2}
    \begin{aligned}
        \sum_{v\in V_{r-1}}d_{\hat{\mh}_{1}}(v)&=\sum_{v\in V_{\geq}}d_{\hat{\mh}_{1}}(v)+\sum_{v\in V_{<}}d_{\hat{\mh}_{1}}(v)
        \leq n_{\ell,r}\binom{4\ell'}{r-1}+|V_{<}|\binom{2\ell'-2}{r-1}.
    \end{aligned}
\end{equation}
By (\ref{equ: (1)}) and (\ref{equ: 2}), we have
\begin{equation*}
\begin{aligned}
    &|\mh_{1}'|=|\mathcal{F}_{C}|+|\hat{\mh}_{1}|
    \leq \sum_{j=1}^{r-2}\frac{|V_{j}|}{r-j}+\binom{4\ell'}{2\ell'-1}\binom{4\ell'}{r-1}+|V_{<}|\binom{2\ell'-2}{r-1}+\binom{4\ell'}{r}+4\ell'\\
    &\leq \sum_{j=1}^{r-2}\frac{|V_{j}|}{r-j}+|V_{r-1}|\binom{2\ell'-2}{r-1}+\binom{4\ell'}{r}+\binom{4\ell'}{2\ell'-1}\binom{4\ell'}{r-1}+4\ell'\\
    &\leq \sum_{j=1}^{r-2}\frac{|V_{j}|}{2}+\left(|V(\hat{\mh}_{1})|-\sum_{j=1}^{r-2}|V_{j}|\right)\binom{2\ell'-2}{r-1}+\binom{4\ell'}{r}+\binom{4\ell'}{2\ell'-1}\binom{4\ell'}{r-1}+4\ell'\\
    &\leq |V(\hat{\mh}_{1})|\binom{2\ell'-2}{r-1}+\binom{4\ell'}{r}+\binom{4\ell'}{2\ell'-1}\binom{4\ell'}{r-1}+4\ell'\\
    &\leq|V(\mh_{1}')|\binom{2\ell'-2}{r-1}+\binom{4\ell'}{r}+\binom{4\ell'}{2\ell'-1}\binom{4\ell'}{r-1}+4\ell'.
\end{aligned}
\end{equation*}
Finally, when $n$ is large enough, we have
\begin{equation*}
    \begin{aligned}
        |\mh|&\leq\left(\binom{2\ell'-1}{r-1}-1\right)|V(\mh)\setminus V(\mh')|+|\mh_{1}'|+|\mh_{2}'|\\
        &\leq \left(\binom{2\ell'-1}{r-1}-1\right)|V(\mh)\setminus V(\mh')|+\binom{2\ell'-2}{r-1}|V(\mh_{1}')|+\frac{|V(\mh_{2}')|}{2\ell'}\binom{2\ell'}{r} \\
        &+\binom{4\ell'}{r}+\binom{4\ell'}{2\ell'-1}\binom{4\ell'}{r-1}+4\ell'\\
        &\leq \left(\binom{2\ell'-1}{r-1}-1\right)n+\binom{4\ell'}{r}+\binom{4\ell'}{2\ell'-1}\binom{4\ell'}{r-1}+4\ell'\\
        &< \binom{2\ell'-1}{r-1}(n-2\ell'+1)+\binom{2\ell'-1}{r}+\binom{2\ell'-1}{r-2},
    \end{aligned}
\end{equation*}
a contradiction with  (\ref{eq-0}).

\noindent
\textbf{Case 2. } $\ell_0= 4\ell'-1$.

By Lemma \ref{2-11}, Theorem \ref{thm: connect Pl} and $\ell\geq r+7$, we have
    \begin{equation*}
        \begin{aligned}
            |\mh_{1}'|&\leq \mathrm{ex}_r^{con}(|V(\mh_{1}')|,\ber P_{\ell_0+1})= \mathrm{ex}_r^{con}(|V(\mh_{1}')|,\ber P_{4\ell'})\\
            &= \binom{2\ell'-1}{r-1}(|V(\mh_{1}')|-2\ell'+1)+\binom{2\ell'-1}{r}+\binom{2\ell'-1}{r-2}.
        \end{aligned}
    \end{equation*}Since $\mh_{2}'$ is $\ber P_{\ell}$-free, by Theorems \ref{thm: turan of BPl} (1), we have
     $|\mh_{2}'|\leq  \frac{|V(\mh_{2}')|}{2\ell'}\binom{2\ell'}{r}.$
   So we have
    \begin{equation*}
        \begin{aligned}
            |\mh|&\leq \left(\binom{2\ell'-1}{r-1}-1\right)|V(\mh)\setminus V(\mh')|+\binom{2\ell'-1}{r-1}|V(\mh_{1}')|+\frac{|V(\mh_{2}')|}{2\ell'}\binom{2\ell'}{r} \\
            &+\binom{2\ell'-1}{r}+\binom{2\ell'-1}{r-2}-(2\ell'-1)\binom{2\ell'-1}{r-1}\\
            &\leq \binom{2\ell'-1}{r-1}(n-2\ell'+1)+\binom{2\ell'-1}{r}+\binom{2\ell'-1}{r-2},
        \end{aligned}
    \end{equation*}
   and we finished the proof of Theorem \ref{thm: berge linear forest} for $k=2$ by combining with (\ref{eq-0}).

\QEDopen

\section{Proof of Theorem \ref{thm: berge linear forest}}
In this section, we prove  Theorem  \ref{thm: berge linear forest}.
We first give some notations.

Let  $V_0\subseteq V(\mh)$ and   $u\in V(\mh)\setminus V_0$.  We call 
$u$  a  {\bf Berge-common neighbour} of $V_0$ if  for every two vertices $v_1,v_2\in V_0$,  there exist two distinct hyperedges $E_1,E_2\in \mh$ such that $\{v_1,u\}\subseteq E_1$ and $\{v_2,u\}\subseteq E_2$. Denote $$BCN(V_0)=\{u\in V(\mh)\setminus V_0: u \mbox{~is a Berge-common neighbour of~} V_0\}.$$
In the following, we set $\ell'=\lchuto$ for convenience. Assume $k\geq 2$, $r\ge 2$, $\ell'\geq r$, $2\ell'\geq r+7$ and $n$ is sufficiently large. For  $V_0\subseteq V(\mh)$ and  $v\in V(\mh)\setminus V_0$, denote $d_{\mh,V_0}(v)=|\{E\in \mh:v\in E,E\cap V_0\not=\emptyset\}|$. Then $d_{\mh}(v)=d_{\mh,V(\mh)}(v)$.

Let $\mh$ be an $r$-uniform $\ber kP_{\ell}$-free hypergraph with size $\mathrm{ex}_r(n,\ber kP_{\ell})$. By (\ref{3-1}), we have
 \begin{align}\label{3-2}
|E(\mh)|&\ge {k\ell'-1\choose r-1}(n-k\ell'+1)+{k\ell'-1\choose r}+\mathbb{I}_\ell\cdot\binom{k\ell'-1}{r-2}\\
&= \binom{k\ell'-1}{r-1}n+O_{r,k,\ell}(1)\nonumber.
\end{align}

We will show that, for $k,r\ge 2$,  $\mathrm{ex}_r(n,\ber kP_{\ell})\le{k\ell'-1\choose r-1}(n-k\ell'+1)+{k\ell'-1\choose r}+\mathbb{I}_\ell\cdot\binom{k\ell'-1}{r-2}$ by induction on $k+r$.
By Theorems \ref{thm:kPl} and  Section 3,
 we may assume $k\geq 3$ and $r\geq 3$. Then $r+k\geq 6$. By Theorem \ref{thm: turan of BPl} (1) and (\ref{3-2}), there exists a copy of $\ber P_{\ell}$ in $\mh$. We first have some lemmas.

 \begin{lemma}\label{lem: enough Berge common}
     There exists a constant $\alpha=\alpha(r,k,\ell)>0$ such that for every copy of Berge-$P_{\ell}$ in $\mathcal{H}$, say $\mathcal{P}$, there is $V_0\subseteq DV_\mh(\mathcal{P})$ with $|V_0|=
     \ell'$ and $|BCN(V_0)|\ge\alpha n$.
 \end{lemma}
 \begin{proof} Let $\mathcal{P}$ be  a copy of $\ber P_{\ell}$ in $\mh$. Set $D=DV_\mh(\mathcal{P})$ for short. Then $|D|=\ell+1$. For $0\le t\le r$, we set
         $$\mathcal{S}_t=\{E\in \mathcal{H}:|E\cap D|=t\}.$$
     Then $\mathcal{H}=\bigcup_{t=0}^{r}\mathcal{S}_t$ and $|\mathcal{S}_r|\leq \binom{\ell+1}{r}$.
     For $0\le t\le r-2$, we set $\mh_t=\{E\setminus D: E\in \mathcal{S}_t\}$ and $DE_\mh(\mathcal{P})|_{\mh_t}=\{E\setminus D: E\in DE_\mh(\mathcal{P})\cap \mathcal{S}_t\}$. Then $\mh_t$ is  $(r-t)$-uniform for $0\le t\le r-2$. Also $\mh_t':=\mh_t-DE_\mh(\mathcal{P})|_{\mh_t}$ is $\ber (k-1)P_{\ell}$-free for $0\le t\le r-2$; otherwise we can find a copy of $\ber kP_{\ell}$ in $\mh$.
    By the induction hypothesis, for $0\le t\le r-2$, we have
     \begin{equation}\label{lemm4-2}
         \begin{aligned}
             |\mh_t|\leq & \mathrm{ex}_{r-t}(n- \ell-1,\ber (k-1)P_{\ell})+|DE_\mh(\mathcal{P})|_{\mh_t}|\\
             \leq & \binom{(k-1)\ell'-1}{r-t-1}(n-(k-1)\ell'+1)+\binom{(k-1)\ell'-1}{r-t}+\binom{(k-1)\ell'-1}{r-t-2}+\ell\\
             \leq &  \binom{(k-1)\ell'-1}{r-t-1}n+O_{r,k,t,\ell}(1).
         \end{aligned}
     \end{equation}
    Note that    $|\mathcal{S}_0|=|\mh_0'|.$ For $1\le t\le r-2$ and $E\in \mh_t'$, define
          \begin{equation*}
         \begin{aligned}
         &C_t(E)=\{E'\in \mathcal{S}_t : E\subseteq E'\},~~c_t(E)=|C_t(E)|,\\
             &\mathcal{E}_{t,1}=\left\{E\in \mh_t' : c_t(E)\geq \binom{\ell'}{t}\right\},\\
             &\mathcal{E}_{t,2}=\left\{E\in \mh_t' : c_t(E)\leq \binom{\ell'}{t}-1\right\}.
         \end{aligned}
     \end{equation*}
     Then we have the following  claim.
     \begin{claim}\label{claim: et1 is small}
        For any $1\le t\le r-2$,  we may assume that $|\mathcal{E}_{t,1}|\leq \eta n$, where  $$ \eta=\frac{1}{2}\min_{t\in\{1,\ldots,r-2\}}\left\{\frac{\binom{(k-1)\ell'-1}{r-t-1}}{\binom{\ell+1}{t}-\binom{\ell'}{t}+1}\right\} >0.$$
     \end{claim}

    \noindent{\bf Proof of Claim \ref{claim: et1 is small}}
    Assume there is $1\le t\le r-2$ such that
 $|\mathcal{E}_{t,1}|>\eta n$. Let
  $E\in \mathcal{E}_{t,1}$ and set $C_t(E)=\{E_{1}',E_{2}',\ldots,E_{c_t(E)}'\}$.
   We construct a hypergraph $\mg_{t,E}=\{E_{j}'\setminus E : j=1,2,\ldots,c_t(E)\}$. Then $\mg_{t,E}$ is a $t$-graph. We use $N_{t,\ell}$ to denote the number of
   $t$-graphs with vertex set $D$. Since $|D|=\ell+1$, $|\mathcal{E}_{t,1}|>\eta n$ and $n$ is large enough, by the Pigeonhole Principle,  there exist $\mathcal{B}_{t,1}\subseteq \mathcal{E}_{t,1}$ with $|\mathcal{B}_{t,1}|\geq \frac{\eta n}{N_{t,\ell}}$ such that for any $E_1,E_2\in\mathcal{B}_{t,1}$ we have $\mg_{t,E_1}=\mg_{t,E_2}$. Note that $\mathcal{B}_{t,1}$ is also a $(r-t)$-hypergraph. Since  $\mh_t'$ is $\ber (k-1)P_{\ell}$-free, $\mathcal{B}_{t,1}$ is $\ber (k-1)P_{\ell}$-free. We have that $|V(\mathcal{B}_{t,1})|\ge
     \frac{\eta n}{2 \binom{(k-1)\ell'-1}{r-t-1}N_{t,\ell}}$; otherwise, by $\mathcal{B}_{t,1}$ being $\ber (k-1)P_{\ell}$-free and  the induction hypothesis, we have
     $$|\mathcal{B}_{t,1}|\leq \mathrm{ex}_{r-t}(|V(\mathcal{B}_{t,1})|,\ber (k-1)P_{\ell})\leq \frac{\eta n}{2 N_{t,\ell}}+O(1)<\frac{\eta n}{N_{t,\ell}},$$
      a contradiction.
      When $t\geq 2$, for each $E\in \mathcal{B}_{t,1}$, we define
     $$S_{t,E}^{0}=\{E'\in \mathcal{S}_{t} : \text{$E\subseteq E'$ and there exist $u_1,u_2\in E'\cap D$ such that $d_{\mg_{t,E}}(u_i)=1,i=1,2$ } \},$$
     and set $S_{1,E}^{0}=\emptyset$. Then we have $|S_{t,E}^{0}|\leq \ell'$ for $t\geq 2$ by $|D|=\ell+1$ and $|S_{1,E}^{0}|=0$.
          For $3\leq t\leq r-2$ and $E\in \mathcal{E}_{t,1}$, we have
     $$|\mg_{t,E}|-|S_{t,E}^{0}|=c_t(E)-|S_{t,E}^{0}|\geq\binom{\ell'}{t}-\ell'\geq \binom{\ell'-1}{t},$$
     which implies $|V(\mg_{t,E}\setminus S_{t,E}^{0})|\ge \ell'-1$ by  $\ell'\ge r$ and $2\ell'\ge r+6$. Then there is $V_0'\subseteq V(\mg_{t,E}\setminus S_{t,E}^{0})$ with $|V_0'|= \ell'-1$ such that $V(\mathcal{B}_{t,1})\subseteq BCN(V_0')$. Since  $|V(\mg_{t,E})|\ge \ell'$ by $|\mg_{t,E}|=c_t(E)\geq \binom{\ell'}{t}$, there is $v_0\in V(\mg_{t,E})\setminus V_0'$ such that $V(\mathcal{B}_{t,1})\subseteq BCN(V_0)$, where $V_0=V_0'\cup\{v_0\}$. Then  $|BCN(V_0)|\ge |V(\mathcal{B}_{t,1})|\ge \frac{\eta n}{2 \binom{(k-1)\ell'-1}{r-t-1}N_{t,\ell}}$.

     Now we consider the case $t=2$. If $|S_{2,E}^{0}|=\ell'$, by the definition of $S_{2,E}^{0}$, $S_{2,E}^{0}$ is a match  in $D$. For each edge in $S_{2,E}^{0}$, we select one endpoint and thus we have  $V_0\subseteq V(S_{2,E}^{0})$ with $|V_0|= \ell'$ such that $V(\mathcal{B}_{2,1})\subseteq BCN(V_0)$. If $|S_{2,E}^{0}|\leq \ell'-1$, we have
     $$|\mg_{2,E}|-|S_{2,E}^{0}|\geq\binom{\ell'}{2}-(\ell'-1)\geq \binom{\ell'-1}{2}\ge \ell'$$ by $\ell'\ge 5$.
     So we have  $V_0\subseteq V(\mg_{2,E}\setminus S_{2,E}^{0})$ with $|V_0|= \ell'$ such that $V(\mathcal{B}_{2,1})\subseteq BCN(V_0)$.

     Finally, we consider the case $t=1$. Since $|\mg_{1,E}|\geq \binom{\ell'}{t}$,  $|V(\mg_{1,E})|\ge \ell'$. Then there is $V_0\subseteq V(\mg_{t,E})$ with $|V_0|= \ell'$ such that $V(\mathcal{B}_{1,1})\subseteq BCN(V_0)$.
     \par
     Above all, set $\alpha=\min_{t=1,\ldots,r-2}\{\frac{\eta}{2\binom{(k-1)\ell'-1}{r-t-1} N_{t,\ell}}\}>0$. If there is $1\le t\le r-2$ such that
 $|\mathcal{E}_{t,1}|>\eta n$, by the  discussion above, Lemma \ref{lem: enough Berge common} holds.\q

Now we complete the proof of Lemma \ref{lem: enough Berge common}.
     By Claim \ref{claim: et1 is small} and (\ref{lemm4-2}), for $1\le t\le r-2$, we have
     \begin{equation}\label{lemma4-1}
     \begin{aligned}
     |\mathcal{S}_{t}|&=\sum_{E\in \mh_t}c_t(E)\leq |\mathcal{E}_{t,1}|\binom{\ell+1}{t}+|\mathcal{E}_{t,2}|\left(\binom{\ell'}{t}-1\right)\\
     & \leq \binom{\ell+1}{t}|\mathcal{E}_{t,1}|+(|\mh_t|-|\mathcal{E}_{t,1}|)\left(\binom{\ell'}{t}-1\right)\\
     &\leq  \binom{\ell+1}{t}|\mathcal{E}_{t,1}|+\left(\binom{(k-1)\ell'-1}{r-t-1}n-|\mathcal{E}_{t,1}|\right)\left(\binom{\ell'}{t}-1\right)+O_{r,k,t,\ell}(1)\\
     &=\left(\binom{\ell+1}{t}-\binom{\ell'}{t}+1\right)|\mathcal{E}_{t,1}|+\binom{(k-1)\ell'-1}{r-t-1}\left(\binom{\ell'}{t}-1\right) n+O_{r,k,t,\ell}(1)\\
     &\leq \left(\binom{\ell+1}{t}-\binom{\ell'}{t}+1\right)\eta n+\binom{(k-1)\ell'-1}{r-t-1}\left(\binom{\ell'}{t}-1\right) n+O_{r,k,t,\ell}(1)\\
     &\leq \binom{(k-1)\ell'-1}{r-t-1}\left(\binom{\ell'}{t}-\frac{1}{2}\right) n+O_{r,k,t,\ell}(1).
     \end{aligned}
     \end{equation}
     Thus when $k\geq 3$, by (\ref{3-2}) and (\ref{lemma4-1}), we have
     \begin{equation}\label{4-41}
         \begin{aligned}
             |\mathcal{S}_{r-1}|&=|\mh|-\sum_{t=0}^{r-2}|\mathcal{S}_t|-|\mathcal{S}_r|\\
             &\geq \binom{k\ell'-1}{r-1}n+O_{r,k,\ell}(1)-\left(\binom{(k-1)\ell'-1}{r-0-1}+\sum_{t=1}^{r-2}\binom{(k-1)\ell'-1}{r-t-1}\left(\binom{\ell'}{t}-\frac{1}{2}\right)\right)n\\
            & -O(r,k,t,\ell)(1)-\binom{\ell+1}{r}\\
            &\geq\left(\binom{k\ell'-1}{r-1}-\sum_{t=0}^{r-2}\binom{(k-1)\ell'-1}{r-t-1}\binom{\ell'}{t}\right)n+\alpha_0 n-O_{r,k,t,\ell}(1),
         \end{aligned}
     \end{equation}
      where $
     \alpha_0=\frac{1}{2}\sum_{t=1}^{r-2}\binom{(k-1)\ell'-1}{r-t-1}>0$.
          Fix a vertex $v\in V(\mh)\setminus D$, we can construct an $(r-1)$-uniform hypergraph $\mathcal{D}_v=\{E\setminus \{v\}: E\in \mathcal{S}_{r-1} \text{~and $v\in E$}\}$. 
    We set $$S_{r-1,v}^0=\{E\in \mathcal{S}_{r-1}: \text{~$v\in E$ and there exist $u_1,u_2\in E\cap D$ with $d_{\mathcal{D}_v}(u_i)=1$, $i=1,2$}\},$$
    and $$\mathcal{S}^+_{r-1,v}=\{E\in \mathcal{S}_{r-1}\setminus\mathcal{S}_{r-1,v}^0:\text{$v\in E$}\}.$$
    For every two vertices $u_1,u_2\in \left(\bigcup_{E\in S^+_{r-1,v}}E\right)\cap D$, there exist two different hyperedges $E_1,E_2\in E(\mathcal{D}_{v})$ with $\{u_i,v\}\subseteq E_i$ for $i=1,2$, which implies $v\in BCN(\{u_1,u_2\})$.

   Denote $\mathcal{S}_{r-1}^0=\bigcup_{v\in V(\mh)\setminus D} S_{r-1,v}^0$ and $\mathcal{S}^+_{r-1}=\bigcup_{v\in V(\mh)\setminus D}\mathcal{S}_{r-1,v}^+$. Then $S^+_{r-1}=\mathcal{S}_{r-1}\setminus S_{r-1}^0$. Note that for every $v\in V(\mh)\setminus D$,  $|S_{r-1,v}^0|\leq \ell'$. Then $|\mathcal{S}_{r-1}^0|\leq \ell' n$ and $|\mathcal{S}_{r-1}^+|\geq |\mathcal{S}_{r-1}|-\ell' n$.

     For a vertex $v\in V(\mh)\setminus D$, we set $$N_{\mathcal{S}_{r-1}^+,D}(v)=\{u\in D: \text{there exist $E\in \mathcal{S}_{r-1}^+$ such that $\{u,v\}\subseteq E$}\},$$ and  $d_{\mathcal{S}_{r-1}^+,D}(v)=|N_{\mathcal{S}_{r-1}^+,D}(v)|$.
    Let $B=\{v\in V(\mh)\setminus D:d_{\mathcal{S}_{r-1}^+,D}(v)\geq \ell'\}$. Then we have
     \begin{equation*}
         \begin{aligned}
             \left|\mathcal{S}_{r-1}^+\right|&\leq \binom{\ell+1}{r-1}|B|+\binom{\ell'-1}{r-1}(n-|B|)\\
             &= \binom{\ell'-1}{r-1}n+\left( \binom{\ell+1}{r-1}-\binom{\ell'-1}{r-1} \right)|B|.
         \end{aligned}
     \end{equation*}
     Note that  $\ell'\geq r\geq 3$ and $k\geq 3$. By (\ref{4-41}) and $|\mathcal{S}_{r-1}^+|\geq |\mathcal{S}_{r-1}|-\ell' n$, we have
      \begin{equation*}
         \begin{aligned}
             &\left( \binom{\ell+1}{r-1}-\binom{\ell'-1}{r-1} \right)|B|\geq \left| \mathcal{S}_{r-1}^+\right|-\binom{\ell'-1}{r-1}n\\
             &\geq \left(\binom{k\ell'-1}{r-1}-\sum_{t=0}^{r-2}\binom{(k-1)\ell'-1}{r-t-1}\binom{\ell'}{t}-\ell'-\binom{\ell'-1}{r-1} \right)n
             +\alpha_0 n-O_{r,k,\ell}(1)   \\
             &= \left(\binom{\ell'-1}{r-2}-\ell'+\alpha_0 \right)n-O_{r,k,\ell}(1).
         \end{aligned}
     \end{equation*}
     By Lemma \ref{lem: the ineuqlaity in sec 3 (2)}, there exists $\beta=\beta(r,k,\ell)>0$ such that $|B|\geq \beta n$.
     Since there are $\binom{\ell+1}{\ell'}$ sets with size $\ell'$ in $D$, there are $B_0\subseteq B$ with  $|B_0|\ge\frac{\beta }{\binom{\ell+1}{\ell'}}n$ and $V_0\subseteq D$ with $|V_0|=\ell'$ such that for every $v\in B_0$ and $u\in V_0$, there exists hyperedges $E\in \mathcal{S}_{r-1}^+$ with $\{u,v\}\subseteq E$.
So for every $v\in B_0$ and $u\in V_0$, $u\in \left(\bigcup_{E\in \mathcal{S}_{r-1,v}^+}E\right)\cap V_0$.
     Then for every two vertices $u_1,u_2\in V_0$ and every vertex $v\in B_0$, we have $u_1,u_2\in  \left(\bigcup_{E\in \mathcal{S}_{r-1,v}^+}E\right)\cap V_0$. Thus $B_0\subseteq BCN(V_0)$.
    By setting $\alpha=\frac{\beta}{\binom{\ell+1}{\ell'}}$, we finished the proof.
 \end{proof}

\begin{lemma}\label{lem: delete one vertex remains many P2l}
     For every vertex $v\in V(\mathcal{H})$, $\mh-v$ contains a copy of $\ber (k-1)P_{\ell}$.
 \end{lemma}
 \begin{proof} Suppose there is $v\in V(\mathcal{H})$ such that
  $\mh-v$ is $\ber (k-1)P_{\ell}$-free. Let $\mathcal{E}_1=E_\mh(v)$ and $\mathcal{E}_2=E(\mh-v)$. Then  $\mh=\mathcal{E}_1\cup \mathcal{E}_2$.
     Since $\mh-v$ is $\ber (k-1)P_{\ell}$-free, by the induction hypothesis,  we have
     \begin{equation*}
          \begin{aligned}
              |\mathcal{E}_2|&\leq \mathrm{ex}_{r}(n-1,\ber(k-1)P_{\ell})\\
              &\leq {(k-1)\ell'-1\choose r-1}(n-(k-1)\ell')+{(k-1)\ell'-1\choose r}+\binom{(k-1)\ell'-1}{r-2}.
          \end{aligned}
      \end{equation*}
      Similarly, we have
      \begin{equation*}
          \begin{aligned}
              |\mathcal{E}_1|&\leq \mathrm{ex}_{r-1}(n-1,\ber kP_{\ell})\\
              &\leq \binom{k\ell'-1}{r-2}(n-k\ell')+\binom{k\ell'-1}{r-1}+\binom{k\ell'-1}{r-3}.
          \end{aligned}
      \end{equation*}
      Thus, when  $k\geq 3$, we have that
     \begin{equation*}
     \begin{aligned}    |\mathcal{H}|=&|\mathcal{E}_1|+|\mathcal{E}_2|\leq \left(\binom{(k-1)\ell'-1}{r-1}+\binom{k\ell'-1}{r-2}\right)n+O_{k,r,\ell}(1)\\
             <& {k\ell'-1\choose r-1}(n-k\ell'+1)+{k\ell'-1\choose r},
         \end{aligned}
     \end{equation*}
     a contradiction with (\ref{3-2}).
     The last inequality holds by Lemma \ref{lem: the ineuqlaity in sec 3 (3)} and when $n$ is large enough.
 \end{proof}

We construct an $\ell'$-uniform hypergraph $\mg=\{H:|BCN(H)|\ge \alpha n,|H|=\ell',H\subseteq V(\mh)\}$, where $\alpha$ is described in Lemma \ref{lem: enough Berge common}. 
We rewrite Lemma \ref{lem: enough Berge common} as follows.
  \begin{lemma}\label{pro: intersecting A with ell}
    For  every copy of $\ber P_{\ell}$ in $\mh$, say $\mathcal{P}$, there is $H\subseteq DV_\mh(\mathcal{P})$  such that  $H\in \mg$. Particularly, $|DV_\mh(\mathcal{P})\cap V(\mg)|\geq \ell'$.
\end{lemma}

 Lemma \ref{lem: delete one vertex remains many P2l} implies that for every vertex $v\in V(\mg)$, we can find $(k-1)$ disjoint hyperedges in $\mg-v$.
Moreover, we have the following lemma.
\begin{lemma}\label{pro: find berge P2l from mg}
    For every $H=\{v_1,\dots,v_{\ell'}\}\in \mg$, there exist a copy of $\ber P_{2\ell'}$ contained $H$ as alternating defining vertices, a copy of $\ber P_{2\ell'-1}$ contained $H$ as alternating defining vertices with one of $v_{1},v_{\ell'}$ as an end vertex, and a copy of $\ber P_{2\ell'-2}$ contained $H$ as alternating defining vertices and  $v_1,v_{\ell'}$ as end vertices.
\end{lemma}
\begin{proof} Let $H=\{v_1,\dots,v_{\ell'}\}\in \mg$. Then $|BCN(H)|\ge \alpha n$. So there exist $w_1\in V(\mh)\setminus H $ and two distinct hyperedges $E_1,E_2$ with $\{w_1,v_1\}\subseteq E_1$ and $\{w_1,v_2\}\subseteq E_2$. When $n$ is large enough, there are $w_2\in V(\mh)\setminus(H\cup E_1\cup E_2)$ and two distinct hyperedges $E_3,E_4$ with $\{v_2,w_2\}\subseteq E_3$ and $\{v_3,w_2\}\subseteq E_4$. Repeat this process until there exist $\omega_{\ell'-1}\in V(\mathcal{H})\setminus (H\cup E_1 \cup E_2\ldots\cup E_{2\ell'-4})$ and two distinct hyperedges $E_{2\ell'-3},E_{2\ell'-2}$ with $\{v_{\ell'-1},\omega_{\ell'-1}\}\subseteq E_{2\ell'-3}$ and $\{v_{\ell'},\omega_{\ell'-1}\}\subseteq E_{2\ell'-2}$. Thus there is a $\ber P_{2\ell'-2}$ with defining vertices $\{v_1,\omega_1,v_2,\ldots,v_{\ell'-1},\omega_{\ell'-1},v_{\ell'}\}$. Since $|BCN(H)|\ge \alpha n$ and $n$ is large enough, there still exist $\omega_0,\omega_{\ell'}\in V(\mathcal{H})\setminus (H\cup E_1\ldots\cup E_{2\ell'-2})$ and two distinct hyperedges ${E_{0}},E_{2\ell'-1}$ with $\{\omega_{0},v_1\}\subseteq E_{0}$ and $\{v_{\ell'},\omega_{\ell'}\}\subseteq E_{2\ell'-1}$. Thus we have a $\ber P_{2\ell'-1}$ with defining vertices $\{\omega_{0},v_1,\omega_1,\ldots,v_{\ell'-1},\omega_{\ell'-1},v_{\ell'}\}$ or $\{v_1,\omega_1,v_2,\ldots,\omega_{\ell'-1},v_{\ell'},\omega_{\ell'}\}$, and also a $\ber P_{2\ell'}$ with defining vertices $\{\omega_{0},v_1,\omega_1,\ldots,\omega_{\ell'-1},v_{\ell'},\omega_{\ell'}\}$.
\end{proof}

By Lemma \ref{pro: find berge P2l from mg},  there are no $k$ disjoint hyperedges in $\mg$ when $n$ is large enough; otherwise we can find a copy of $\ber kP_{\ell}$. Then we can bound the size of $V(\mg)$ by a constant relative to $k,r$ and $\ell$.

\begin{lemma}\label{lem: A is bounded}
     There exists a constant $C=C(k,r,\ell)>0$ such that $|V(\mg)|\leq C$.
 \end{lemma}
\begin{proof}
    Let $\Delta(\mg)=\max_{v\in V(\mg)}d_\mg(v)$. Then we have the following result.
    \vskip.2cm
    \begin{claim}\label{claim: max degree}
       If there exists a constant $\hat{C}=\hat{C}(r,k,\ell)>0$ such that $\Delta (\mg)\leq \hat{C}$, then there exists $C=C(r,k,\ell)$ with $|V(\mg)|\leq C$.
    \end{claim}

     \noindent{\bf Proof of Claim \ref{claim: max degree}}
     For any vertex $v\in V(\mg)$, by Lemma \ref{lem: delete one vertex remains many P2l}, there are $k-1$ disjoint hyperedges in $\mg-v$, denote them as $H_1,\dots,H_{k-1}$.
     By Lemma \ref{pro: find berge P2l from mg},  there are no $k$ disjoint hyperedges in $\mg$ which implies  for all $E\in \mg$ there is $1\le i\le k-1$ such that $E\cap H_i\not=\emptyset$.
     So we have
    $$|V(\mg)|\leq \sum_{v\in \bigcup_{i=1}^{k-1}H_i}\ell'd_\mg(v)\leq (k-1)\ell'^2\Delta(\mg)\leq (k-1){\ell'}^2\hat{C}.$$
    By setting $C=(k-1)\ell'^2\hat{C}$, the claim holds. \q

    Let $k_q=k$ and $k_i=\frac{2}{\alpha}+2(k_{i+1}-1)+1$ for $i=1,2,\ldots,q-1$, where $q=\lceil\log_{2}(\ell+1)\rceil+1$ and
    $\alpha$ is the constant described in Lemma \ref{lem: enough Berge common}.

\vskip.2cm
\begin{claim}\label{claim: 3 is true}
    $\Delta(\mg)\leq \binom{k_1-1}{r-1}$.
\end{claim}

    \noindent{\bf Proof of Claim \ref{claim: 3 is true}}
    Suppose there       exists  $v\in V(\mg)$ with $d_\mg(v)>\binom{k_1-1}{r-1}$. First we can observe the following two Facts.
   \par\noindent
   \textbf{Fact 1.} Let $\{H_1,H_{2},\ldots,H_{s}\}\subseteq \mg$ and  $S\subseteq V(\mh)$ such that $|S|<\frac{\alpha}{2}n$. If $s>\frac{2}{\alpha}$, then there exist  $u\in V(\mh)\setminus S$ and $i,j\in\{1,\ldots,s\}$ with $i\not=j$, say $i=1,j=2$, such that $u\in BCN(H_1)\cap BCN(H_2)$. (Otherwise $n\ge\sum_{i=1}^s|BCN(H_i)\cap (V(\mh)\setminus S)|>n$, a contradiction.)
   \par\noindent
   \textbf{Fact 2}: Assume $u$ and $H_1,H_2$ are defined as Fact 1. Let $v_1\in H_1\setminus H_2$ and $v_2\in H_2\setminus H_1$. Then either there exist distinct hyperedges $E_1,E_2\in E(\mh)$ such that $\{v_1,u\}\subseteq E_1$ and $\{v_2,u\}\subseteq E_2$ or for any hyperedge $E\in E(\mh)$, if $\{v_i,u\}\subseteq E$ then $v_{3-i}\in E$, where $i=1,2$.
   \par\noindent

   Now we will find a $\ber kP_{\ell}$ in $\mh$ by Facts 1 and 2. Since $d_\mg(v)>\binom{k_1-1}{r-1}$, by Lemma \ref{lem: berge star}, there is a $\ber S_{k_1,v}$ in $\mg$, whose the defining hyperedges are $DH=\{H_1,H_2,\cdots,H_{k_1}\}$ and the defining vertices are $DS=\{v,v_{1,1},v_{1,2},\cdots,v_{1,k_1}\}$ with $v_{1,i}\in V(H_i)$. We will use $P_{\ge s}(v_{1,i},v_{1,j})$ to denote a Berge path of length at least $s$ with two end points $v_{1,i}$ and $v_{1,j}$.

   \par \noindent
   \textbf{Step $1$}:
      Since $k_1>\frac{2}{\alpha}$, by Fact 1, we assume there is $v_{2,1}\in V(\mh)\setminus DS $ such that $v_{2,1}\in BCN(H_1)\cap BCN(H_2)$. By Fact 2, there are $E_{2,1}\in E(\mh)$ such that $\{v_{1,1},v_{1,2},v_{2,1}\}\subseteq E_{2,1}$ or  $E_{2,1},E_{2,2}\in E(\mh)$ such that $\{v_{2,1},v_{1,1}\}\subseteq E_{2,1}$ and $\{v_{2,1},v_{1,2}\}\subseteq E_{2,2}$. Thus we have a Berge path $P_{\ge1}^{1}(v_{1,1},v_{1,2})$.
      \par
      Since $k_1-2>\frac{2}{\alpha}$, by Fact 1, there is $v_{2,3}\in V(\mh)\setminus (DS\cup V(P_{\ge1}^{1}(v_{1,1},v_{1,2})))$ such that $v_{2,3}\in BCN(H_3)\cap BCN(H_4)$. By Fact 2, there are $E_{2,3}\in E(\mh)$ such that $\{v_{1,3},v_{1,4},v_{2,3}\}\subseteq E_{2,3}$ or  $E_{2,3},E_{2,4}\in E(\mh)$ such that $\{v_{2,3},v_{1,3}\}\subseteq E_{2,3}$ and $\{v_{2,3},v_{1,4}\}\subseteq E_{2,4}$. So we have a Berge path $P_{\ge1}^{2}(v_{1,3},v_{1,4})$.
      \par
     Note that
      $k_1-2(k_2-1)>\frac{2}{\alpha}$.   Repeat this process,  there exists $v_{2,2k_2-1}\in V(\mh)\setminus (DS\cup \left(\cup_{i=1}^{k_2-1}V(P_{\ge 1}^{i}(v_{1,2i-1},v_{1,2i}))\right))$ and  $E_{2,2k_2-1}\in E(\mh)$ such that $\{v_{1,2k_2-1},v_{1,2k_2},v_{2,2k_2-1}\}\subseteq E_{2,2k_2-1}$ or  $E_{2,2k_2-1},E_{2,2k_2}\in E(\mh)$ such that $\{v_{2,2k_2-1},v_{1,2k_2-1}\}\subseteq E_{2,2k_2-1}$ and $\{v_{2,2k_2-1},v_{1,2k_2}\}\subseteq E_{2,2k_2}$. We find a Berg path $P_{\ge1}^{k_2}(v_{1,2k_2-1},v_{1,2k_2})$.
    \par

    Let $DS_1=DS\cup \cup_{i=1}^{k_2}V(P_{\ge 1}^{i}(v_{1,2i-1},v_{1,2i}))$. Then $|DS_1|\le k_1+1+2rk_2$ and there is  a $\ber k_2P_{\ge1}$ in $\mh$, where the $k_2$ disjoint copies of $\ber P_{\ge1}$ are $P_{\ge 1}^{i}(v_{1,2i-1},v_{1,2i})$ for $1\le i\le k_2$.

    \par\noindent
    \textbf{Step $2$}:
      Consider $\{H_{1},H_{3},\ldots,H_{2k_2-1}\}$. Since $k_2>\frac{2}{\alpha} $, by Fact 1, we assume there exists $v_{3,1}\in V(\mh)\setminus DS_1$ such that $v_{3,1}\in BCN(H_1)\cap BCN(H_3)$.
      By the same argument as Step 1, there is $P_{\ge 1}(v_{1,1},v_{1,3})$ with $ v_{3,1}\in V(P_{\ge 1}(v_{1,1},v_{1,3}))$. Then we have $P_{\ge 3}^1(v_{1,2},v_{1,4}):=P_{\ge1}^{1}(v_{1,1},v_{1,2})P_{\ge 1}(v_{1,1},v_{1,3})P_{\ge1}^{2}(v_{1,3},v_{1,4})$.

       Note that $k_2-2(k_3-1)>\frac{2}{\alpha}$. Repeat this process,   there exist $v_{3,4k_3-3}\in V(\mh)\setminus (DS_1\cup \left(\cup_{i=1}^{k_3-1}V(P_{\ge 3}^{i}(v_{1,4i-2},v_{1,4i}))\right))$ such that $v_{3,4k_3-3}\in BCN(H_{4k_3-3})\cap BCN(H_{4k_3-1})$. By the same argument as Step 1, there is $P_{\ge 1}(v_{1,4k_3-3},v_{1,4k_3-1})$ with $v_{3,4k_3-3}\in V(P_{\ge 1}(v_{1,4k_3-3},v_{1,4k_3-1}))$. Then we have $$P_{\ge 3}^{k_3}(v_{1,4k_3-2},v_{1,4k_3}):=P_{\ge 1}^{2k_3-1}(v_{1,4k_3-3},v_{1,4k_3-2})P_{\ge 1}(v_{1,4k_3-3},v_{1,4k_3-1})P_{\ge 1}^{2k_3}(v_{1,4k_3-1},v_{1,4k_3}).$$
          \par
     Let $DS_2=DS_1\cup \left(\cup_{i=1}^{k_3}V(P_{\ge 1}^{i}(v_{1,4i-2},v_{1,4i}))\right)$. Then $|DS_2|\le k_1+1+2rk_2+2rk_3$ and there is  a $\ber k_3P_{\ge3}$ in $\mh$, where the $k_3$ disjoint copies of $\ber P_{\ge3}$ are $P_{\ge 3}^{i}(v_{1,4i-2},v_{1,4i})$ for $1\le i\le k_3$.

     \par\noindent
     {\bf Step $j$ $(3\le j\le q-1)$:}  Consider $\{H_{2^{j-2}},H_{2^{j-2}+2^{j-1}},\ldots,H_{2^{j-2}+2^{j-1}(k_j-1)}\}$. Since $k_j>\frac{2}{\alpha}$, by Fact 1, we assume there exists $v_{j+1,2^{j-2}}\in V(\mh)\setminus DS_{j-1}$ such that $v_{j+1,2^{j-2}}\in BGN(H_{2^{j-2}})\cap BCN(H_{2^{j-2}+2^{j-1}})$.
     By the same argument as Step 1, there is $P_{\ge 1}(v_{1,2^{j-2}},v_{1,2^{j-2}+2^{j-1}})$ with $v_{j+1,2^{j-2}}\in V(P_{\ge 1}(v_{1,2^{j-2}},v_{1,2^{j-2}+2^{j-1}}))$. Then we have a Berge path
     \par\noindent
     $P_{\ge 2^j-1}^{1}(v_{1,2^{j-1}},v_{1,2^j}):=P_{\ge 2^{j-1}-1}^{1}(v_{1,2^{j-2}},v_{1,2^{j-1}})P_{\ge 1}(v_{1,2^{j-2}},v_{1,2^{j-2}+2^{j-1}})P_{\ge 2^{j-1}-1}^{2}(v_{1,2^{j-2}+2^{j-1}},v_{1,2^j})$.
     \par
     Note that $k_j-2(k_{j+1}-1)>\frac{2}{\alpha}$.
     Repeat this process,  there exist $v_{j+1,2^jk_{j+1}-3\cdot 2^{j-2}}\in V(\mh)\setminus \left(DS_{j-1}\cup \left(\cup_{i=1}^{k_{j+1}-1}V(P_{\ge 2^j-1}^{i}(v_{1,2^{j}i-2\cdot2^{j-2}},v_{1,2^{j}i}))\right)\right)$ such that $$v_{j+1,2^jk_{j+1}-3\cdot 2^{j-2}}\in BCN(H_{2^{j}k_{j+1}-3\cdot2^{j-2}})\cap BCN(H_{2^{j}k_{j+1}-2^{j-2}}).$$ By the same argument as Step 1, there is $P_{\ge1}(v_{1,2^{j}k_{j+1}-3\cdot2^{j-2}},v_{1,2^{j}k_{j+1}-2^{j-2}})$ with $v_{j+1,2^jk_{j+1}-3\cdot 2^{j-2}}\in V(P_{\ge1}(v_{1,2^{j}k_{j+1}-3\cdot2^{j-2}},v_{1,2^{j}k_{j+1}-2^{j-2}}))$. Then we have
    \begin{equation*}
    \begin{aligned}
     &P_{\ge 2^j-1}^{k_{j+1}}(v_{1,2^{j}k_{j+1}-2\cdot2^{j-2}},v_{1,2^{j}k_{j+1}})
     :=P_{\ge 2^{j-1}-1}^{2k_{j+1}-1}(v_{1,2^{j}k_{j+1}-3\cdot2^{j-2}},v_{1,2^{j}k_{j+1}-2\cdot2^{j-2}})\\
     ~~~~~~~~~~~~~~~&P_{\ge1}(v_{1,2^{j}k_{j+1}-3\cdot2^{j-2}},v_{1,2^{j}k_{j+1}-2^{j-2}})
     P_{\ge2^{j-1}-1}^{2k_{j+1}}(v_{1,2^{j}k_{j+1}-2^{j-2}},v_{1,2^{j}k_{j+1}}).
    \end{aligned}
\end{equation*}

     \par
     Let $DS_{j}=DS_{j-1}\cup \left(\cup_{i=1}^{k_{j+1}}V(P_{\ge 2^j-1}^{i}(v_{1,2^{j}i-2\cdot2^{j-2}},v_{1,2^{j}i}))\right)$. Then $|DS_j|\le k_1+1+2rk_2+\cdots+2rk_{j+1}$ and there is a $\ber k_{j+1}P_{\ge 2^{j}-1}$ in $\mh$.

\par

     After Step $q-1$, there is a $\ber k_{q}P_{\ge2^{q-1}-1}$ in $\mh$. From which we find a $\ber kP_{\ell}$, a contradiction. Thus $\Delta(\mg)\leq \binom{k_1-1}{r-1}$.
    \q

 By Claims \ref{claim: max degree} and \ref{claim: 3 is true}, we finished the proof.
\end{proof}


For  $V_0\subseteq V(\mh)$ and  $v\in V(\mh)\setminus V_0$, recall $d_{\mh,V_0}(v)=|\{E\in E(\mh):v\in E,E\cap V_0\not=\emptyset\}|$. Set
\begin{equation*}
    \begin{aligned}
        X=&\{v\in V(\mh)\setminus V(\mg):d_{\mh,V(\mg)}(v)\geq 2\},\\
        Y=&\{v\in V(\mh)\setminus V(\mg):d_{\mh,V(\mg)}(v)\leq 1\}.
    \end{aligned}
\end{equation*}
Then $\mh[Y]$ is $\ber P_{\ell}$-free by Lemma \ref{pro: intersecting A with ell} and $\ell\ge 9$. By Theorem \ref{thm: turan of BPl} (1), we have
\begin{align}\label{eq4-1}|E(\mh[Y])|\leq \frac{1}{\ell}\binom{\ell}{r}|Y|.\end{align}
We define another hypergraph $\mh_1$ with 
$$\mh_1=\{S\subseteq X\cup Y: \text{$|S|\geq 2$, $S\subseteq E$ for some $E\in \mh$ and $E\setminus S\subseteq V(\mg)$}\}.$$
For each $S\in \mh_1$, we have $2\le |S|\le r$.
Then we have the following result.
\begin{lemma}\label{lem: no Berge P3}
    For every $S_1,S_2\in \mh_1$ with $S_1\cap S_2\neq \emptyset$, we have $(S_1\Delta S_2)\cap X=\emptyset$, where $S_1\Delta S_2=(S_1\setminus S_2)\cup (S_2\setminus S_1)$.
\end{lemma}
\begin{proof} Suppose there are $S_1,S_2\in \mh_1$ with $S_1\cap S_2\neq \emptyset$ such that $(S_1\Delta S_2)\cap X\not=\emptyset$.
    Let $v_1\in S_1\cap S_2$ and $v_2\in (S_1\Delta S_2)\cap X$, say $v_2\in S_2\setminus S_1$. Assume $S_i\subseteq E_i$ for some $E_i\in \mh$ and $i=1,2$. Since  $d_{\mh,V(\mg)}(v_2)\geq 2$, there exist $u\in V(\mg)$ and  $E\in E(\mh)\setminus\{E_1,E_2\}$ with $\{u,v_2\}\subseteq E$. Since $|S_1|\ge 2$, there is $v_0\in S_1\setminus\{v_1\}$. Then $\{v_0,E_1,v_1,E_2,v_2,E,u\}$ is a $\ber P_{3}$ with defining vertices $\{v_0,v_1,v_2,u\}$.
    Assume $u\in H$ for some $H\in \mg$ and $H=\{u,w_2,\dots,w_{\ell'}\}$. By Lemma \ref{pro: find berge P2l from mg}, there is a copy of $\ber P_{2\ell'-1}$ contained $H$ as alternating defining vertices, and $u$ is an end vertices. Together with the $\ber P_{3}$, there is a copy of $\ber P_{2\ell'+2}$, denoted as $\mathcal{P}$. Then we have a $\ber P_{\ell}$, say $\mathcal{P}'$, with $V(\mathcal{P}')\subseteq V(\mathcal{P})$ but $|DV_\mh(\mathcal{P}')\cap V(\mg)|\le \ell'-1$,
     a contradiction with Lemma \ref{pro: intersecting A with ell}.
\end{proof}
\begin{lemma}\label{lem:E(X,Y)}
    For every $v\in Y$, $d_{\mh,X}(v)\leq 1$.
\end{lemma}
 \begin{proof}
     Suppose there exists $v\in Y$ such that $d_{\mh,X}(v)\geq 2$. Let $E_1,E_2\in \mh$ with $v\in E_1\cap E_2$ and $E_i\cap X\not=\emptyset$, $i=1,2$.
     \par
          Since $v\in Y$, we can assume $E_2\cap V(\mg)=\emptyset$. Assume  $v_1\in E_1\cap X$. Then  there exists  $E\in \mh\setminus\{E_1,E_2\}$ and  $u\in V(\mg)$ such that $\{u,v_1\}\subseteq E$. Assume $u\in H$ for some $H\in \mg$ and $H=\{u,w_2,\dots,w_{\ell'}\}$. By Lemma \ref{pro: find berge P2l from mg}, there is a copy of $\ber P_{2\ell'-1}$ contained $H$ as alternating defining vertices, and $u$ is an end vertex. Together with $E, E_1$ and $E_2$, there is a copy of $\ber P_{\ell}$, say $\mathcal{P}'$, with $V(\mathcal{P}')\subseteq V(\mathcal{P})$ but $|DV_\mh(\mathcal{P}')\cap V(\mg)|\le \ell'-1$,
     a contradiction with Lemma \ref{pro: intersecting A with ell}.
  \end{proof}

\par
For $0\le t \le r$, let $$\mathcal{T}_t:=\{E\in \mh[V(\mg)\cup X]: |E\cap V(\mg)|=t\}.$$

\noindent Now, we consider subsets of $\mh_1$ described as follows.
\begin{equation*}
    \begin{aligned}
        \mathcal{U}_{2}=&\{S\subseteq X: \text{$|S|=2$ , $S\subseteq E$ for some $E\in \mh[V(\mg)\cup X]$ and $E\setminus S\subseteq V(\mg)$}\},\\
         \mathcal{U}_{3+}=&\{S\subseteq X: \text{$|S|\geq 3$ , $S\subseteq E$ for some $E\in \mh[V(\mg)\cup X]$ and $E\setminus S\subseteq V(\mg)$}\}.
    \end{aligned}
\end{equation*}
\noindent Then, we have the following property.
\begin{lemma}\label{pro: degree of 3-set is at most 2}
    For every $t$-set $S\in \mathcal{U}_{3+}$, $|\{E\in \mathcal{T}_{r-t}:S\subseteq E \}|\leq 2$,  where $3\leq t\leq r$.
\end{lemma}
\begin{proof} Suppose there a $t$-set $S\in \mathcal{U}_{3+}$ such that $|\{E\in \mathcal{T}_{r-t}:S\subseteq E \}|\geq 3$. By the same argument as the proof of
     Lemma \ref{lem: no Berge P3}, there is a copy of $\ber P_{\ell}$ that contains only $\ell'-1$ defining vertices in $V(\mg)$, a contradiction with the Lemma \ref{pro: intersecting A with ell}.
\end{proof}
Now, we partition  $\mathcal{U}_2$ into two subsets.
\begin{equation*}
    \begin{aligned}
        \mathcal{U}_2^-=&\{S\in \mathcal{U}_2: |\{E\in \mathcal{T}_{r-2}:S\subseteq E\}|\leq 2\}~ \text{~and}\\
         \mathcal{U}_2^+=&\{S\in \mathcal{U}_2: |\{E\in \mathcal{T}_{r-2}:S\subseteq E\}|\geq 3\}.
    \end{aligned}
\end{equation*}
Correspondingly, we partition $\mathcal{T}_{r-2}$ into two sets.
\begin{equation*}
    \begin{aligned}
        \mathcal{T}_{r-2}^-=&\{E\in \mathcal{T}_{r-2}: E\cap X\in \mathcal{U}_2^- \} ~\text{~and}\\
        \mathcal{T}_{r-2}^+=&\{E\in \mathcal{T}_{r-2}: E\cap X\in \mathcal{U}_2^+ \}.
    \end{aligned}
\end{equation*}
By Lemma \ref{lem: no Berge P3}, the hyperedges in $\mathcal{U}:=\mathcal{U}_2\cup \mathcal{U}_{3+}$ are pairwise disjoint. Thus $|\mathcal{U}|\leq \frac{|X|}{2}$. By Lemma \ref{pro: degree of 3-set is at most 2}, we have
$$\left|\left(\cup_{t=0}^{r-3} \mathcal{T}_t\right)\cup \mathcal{T}_{r-2}^-\right|\leq 2|\mathcal{U}|\leq |X|.$$
Next, we will bound the sets $\mathcal{U}_2^+$.
\begin{lemma}\label{lem: U_2^+}
    We have $\left|\mathcal{U}_2^+\right|\leq \binom{|V(\mg)|}{r-2}$.
\end{lemma}
\begin{proof}
 Suppose $|\mathcal{U}_{2}^{+}|\geq \binom{|V(\mg)|}{r-2}+1$. Then there exist $S_1,S_2\in \mathcal{U}_{2}^{+}$ and  $E_1,E_2\in \mathcal{T}_{r-2}$ such that $S_1\subseteq E_1, S_2\subseteq E_2$ and $E_1\setminus S_1=E_2\setminus S_2$. Since $|S_1|=|S_2|=2$ and $r\geq 3$, there exist a vertex $u\in E_1\setminus S_1$. Since $S_1\in\mathcal{U}_{2}^{+}$, there exist  $E_3,E_4\in \mathcal{T}_{r-2}\setminus\{E_1\}$ such that $S_1\subseteq E_3$   and $S_1\subseteq E_4$. Then there exist  $u_1\in (E_4\cap V(\mg))\setminus\{u\}$ and  $H\in E(\mg)$ such that $u_1\in H$.
Let $x_1,x_2\in S_1$ and $y\in S_2\setminus S_1$. Then $\{y,E_2,u,E_1,x_1,E_3,x_2,E_4,u_1\}$ is a  $\ber P_{4}$ with defining vertices $\{y,x_1,x_2,u,u_1\}$.
  By the same argument as the proof of Lemma \ref{pro: find berge P2l from mg}, there is a copy of $\ber P_{2\ell'-4}$ contained $u_1$ as an end vertex and $\{u_1,\ldots,u_{\ell'-1}\}\subseteq H\setminus\{u\}$ as its defining vertices. Together with the $\ber P_{4}$, we have a copy of $\ber P_{\ell}$, say $\mathcal{P}'$, such that $|DV_\mh(\mathcal{P}')\cap V(\mg)|=\ell'-1$, a contradiction with Lemma \ref{pro: intersecting A with ell}.
\end{proof}

Since $\left|\mathcal{T}_{r-2}^+\right|\leq \binom{|V(\mg)|}{r-2}\left|\mathcal{U}_2^+\right|$, by Lemma \ref{lem: U_2^+}, we have
$\left|\mathcal{T}_{r-2}^+\right|\leq \binom{|V(\mg)|}{r-2}^2.$
Then, the above inequalities implies that
\begin{equation}\label{eq: number of edges 1}
    \begin{aligned}
|E(V(\mg),X)|+|\mg|+|\mh[X]|=&|\bigcup_{t=0}^{r-3} \mathcal{T}_t|+\left|\mathcal{T}_{r-2}^-\right|+\left|\mathcal{T}_{r-2}^+\right|+\left|\mathcal{T}_{r-1}\right|+|\mg|\\
        \leq & |X|+\binom{|V(\mg)|}{r-2}^2+\binom{|V(\mg)|}{r}+|\mathcal{T}_{r-1}|.
    \end{aligned}
\end{equation}
Define $E[V(\mg),Y]=\{E\in E(\mh):\text{$E\cap V(\mg)\neq \emptyset$ and $E\cap Y\neq \emptyset $} \}$. We have
\begin{equation}\label{eq: E(A,Y)}
    \begin{aligned}
        |E[V(\mg),Y]|\leq |Y|.
    \end{aligned}
\end{equation}
For every vertex $v\in X$, let
$X^+=\left\{v\in X: d_{\mathcal{T}_{r-1}}(v)\geq \binom{k\ell'-1}{r-1}-\frac{1}{2}\binom{k\ell'-1}{r-2}\right\}.$
 Then we have the following result.
\begin{lemma}\label{55-1}
    When $n$ is large enough, $|X^+|\geq \frac{1}{2}\binom{|V(\mg)|}{r-1}^{-1}n$.
\end{lemma}
\begin{proof} Suppose  $|X^+|< \frac{1}{2}\binom{|V(\mg)|}{r-1}^{-1}n$.
    By Lemma \ref{lem:E(X,Y)}, $|E(X,Y)|\leq |Y|$.  By (\ref{eq4-1}),   (\ref{eq: number of edges 1}) and (\ref{eq: E(A,Y)}), we have
    \begin{equation*}
        \begin{aligned}
    |\mh|=&|E(V(\mg),X)|+|\mg|+|\mh[X]|+|E(X,Y)|+|\mh[Y]|+|E[V(\mg),Y]|\\
            \leq & |X|+\binom{|V(\mg)|}{r-2}^2+\binom{|V(\mg)|}{r}+|\mathcal{T}_{r-1}|+\left(\frac{1}{\ell}\binom{\ell}{r}+2\right)|Y|.
        \end{aligned}
    \end{equation*}
    By Lemma \ref{lem: A is bounded}, $|V(\mg)|\leq C$, where $C=C(k,r,\ell)>0$
     is a constant.
Since $|\mathcal{T}_{r-1}|\leq \binom{|V(\mg)|}{r-1}|X^+|+(|X|-|X^+|)\left(\binom{k\ell'-1}{r-1}-\frac{1}{2}\binom{k\ell'-1}{r-2}-1\right)$,  we have
\begin{equation*}
    \begin{aligned}
        |\mh| &\leq  |X|+\binom{|V(\mg)|}{r-1}|X^+|+(|X|-|X^+|)\left(\binom{k\ell'-1}{r-1}-\frac{1}{2}\binom{k\ell'-1}{r-2}-1\right)\\
        &+\left(\frac{1}{\ell}\binom{\ell}{r}+2\right)|Y|+O_{r,k,\ell}(1)\\
        &< \frac{1}{2}n+\left(\binom{k\ell'-1}{r-1}-\frac{1}{2}\binom{k\ell'-1}{r-2}\right)|X|+\left(\frac{1}{\ell}\binom{\ell}{r}+2\right)|Y|+O_{r,k,\ell}(1)\\
        &\leq \frac{1}{2}n+\max\left\{\binom{k\ell'-1}{r-1}-\frac{1}{2}\binom{k\ell'-1}{r-2},\frac{1}{\ell}\binom{\ell}{r}+2\right\}n+O_{r,k,\ell}(1)\\
        &<\binom{k\ell'-1}{r-1}(n-k\ell'+1),
    \end{aligned}
\end{equation*}
a contradiction with (\ref{3-2}). The last inequality holds by Lemma \ref{last}.\end{proof}

By Lemma \ref{55-1}, $X^+\not=\emptyset$.
Then for every $v\in X^+$, we have $d_{\mh,V(\mg)}(v)\geq k\ell'-1$; otherwise $d_{\mathcal{T}_{r-1}}(v)\leq \binom{k\ell'-2}{r-1}<\binom{k\ell'-1}{r-1}-\frac{1}{2}\binom{k\ell'-1}{r-2}$, a contradiction with the definition of $X^{+}$. Thus
$|V(\mg)|\geq k\ell'-1$.
For every vertex $v\in X^+$,  define $\mg_v=\{E\setminus\{v\}: E\in \mathcal{T}_{r-1} \text{~and~} v\in E\}$. Then $\mg_v$ is  $(r-1)$-uniform and $V(\mg_v)\subseteq V(\mg)$. Let $M_{r,k,\ell}=|\{\mg_v: k\ell'-1\le |V(\mg_v)|\le |V(\mg)|,v\in X^+\}|$.

By Lemmas \ref{lem: A is bounded} and \ref{55-1},
  there exist a subset $X^{++}\subseteq X^+$ with  $|X^{++}|\ge\frac{1}{2}\binom{|V(\mg)|}{r-1}^{-1}M_{r,k,\ell}^{-1}n$ and an $(r-1)$-uniform hyperedges $\mathcal{G}_0$ such that $\mathcal{G}_v=\mathcal{G}_0$ for every $v\in X^{++}$, where $k\ell'-1\leq |V(\mg_0)|\leq |V(\mg)|$.

Now, we consider the hypergraph $\mathcal{G}_0$.
 Recall that two vertices $u,v\in V(\mg_0)$ is a good pair if there exist two hyperedges $E_1,E_2\in E(\mg_0)$ such that $u\in E_1$ and $v\in E_2$. 
 We have the following result.
 \begin{lemma}\label{lem: bound V(G0)}
     We have $|V(\mg_0)|=k\ell'-1$.
 \end{lemma}
 \begin{proof}
     Suppose $|V(\mg_0)|\geq k\ell'$. By Lemma \ref{lem: good order}, let $V(\mg_0)=\{v_1,\dots,v_{|V(\mg_0)|}\}$ be  a good order of $V(\mg_0)$. Then for every  vertex $w\in X^{++}$, $w\in BCN(\{v_i,v_{i+1}\}$ for $i=1,\dots,|V(\mg_0)|-1$. Since $|X^{++}|\geq \frac{1}{2}\binom{|V(\mg)|}{r-1}M_{r,k,\ell}^{-1}n$, when $n$ is large enough, we can find a copy of $\ber P_{\ell}$ (denoted as $\mathcal{P}_1$) such that $DV_\mh(\mathcal{P}_1)\cap V(\mg_0)=\{v_1,\dots,v_{\ell'}\}$, and stepwise, we can find $\mathcal{P}_i$ disjoint with $\mathcal{P}_1,\dots, \mathcal{P}_{i-1}$ and $DV_\mh(\mathcal{P}_i)\cap V(\mg_0)=\{v_{(i-1)\ell'+1},\dots,v_{i\ell'}\}$. Since  $|V(\mg_0)|\geq k\ell'$, there are $k$ disjoint copies of $\ber P_{\ell}$, a contradiction.
 \end{proof}
\begin{lemma}\label{4.12}
    We have $|V(\mg)|=k\ell'-1$.
\end{lemma}
\begin{proof}
    It's sufficent to prove $V(\mg_0)=V(\mg)$. Suppose there exists $v\in V(\mg)\setminus V(\mg_0)$. Assume $v\in H$, where $H\in \mg$.
     Suppose $H\cap V(\mg_0)=\emptyset.$ By
      the same process as  the proof of Lemma \ref{lem: bound V(G0)}, there are $(k-1)$ disjoint copies of $\ber P_{\ell}$ in $\mh\setminus \{H\}$. By Lemma \ref{pro: find berge P2l from mg}, we can find another copy of $\ber P_{\ell}$ disjoint with the $k-1$ copies of $\ber P_{\ell}$ from $H$, a contradiction. Let $u\in H\cap V(\mg_0)$. By  Lemma \ref{lem: good order}, we can let $V(\mg_0)=\{v_1(=u),v_2,\dots,v_{k\ell'-1}\}$ be a good order of $V(\mg_0)$. Since $u,v$ has large enough Berge-common neighbours, we can find a copy of $\ber P_{\ell}$ contains $\{v,v_1(=u),\dots,v_{\ell'-1}\}$ as defining vertices. By the same argument as the proof of Lemma \ref{lem: bound V(G0)}, there are $k$ disjoint copies of $\ber P_{\ell}$ in $\mh$, a contradiction.
\end{proof}

Let $\mathcal{U}^{X^+}=\{S\in \mathcal{U}: S\cap X^+\neq \emptyset\}$ and  $\mathcal{U}[X\setminus X^+]=\{S\in
\mathcal{U}:S\subseteq X\setminus X^+\}$.
Then $\mathcal{U}=\mathcal{U}^{X^+}\cup \mathcal{U}[X\setminus X^+]$.
\begin{lemma}\label{lem: U X+ is small}
   We have $|\mathcal{U}^{X^+}|\leq 1$. Moreover, when $\ell$ is odd, $\mathcal{U}^{X^+}=\emptyset$.
\end{lemma}
\begin{proof}
Suppose there are  $S_1,S_2
\in \mathcal{U}^{X^+}$, say $S_i\subseteq E_i$ for some $E_i\in \mh$ and $i=1,2$. By Lemma \ref{lem: no Berge P3}, $S_1\cap S_2=\emptyset$. Assume  $v_i\in S_i\cap X^+$ for $i=1,2$. Obviously, $E_1,E_2\notin \mathcal{T}_{r-1}$. Since $v_1\in X^{+}\subseteq X$, there exist  $E_{3}\in \mathcal{T}_{r-1}$ and  $u_1\in V(\mg)$ such that $\{u_1,v_1\}\subseteq E_3.$ Let $H\in \mg$ with $u_1\in H$. Since $|V(\mg)|=k\ell'-1$,  $|V(\mg)\setminus (H\setminus \{u_1\})|=(k-1)\ell'$. Since $v_2\in S_2\cap X^{+}$,   $d_{\mathcal{T}_{r-1}}(v_2)\geq \binom{k\ell'-1}{r-1}-\frac{1}{2}\binom{k\ell'-1}{r-2}>\binom{(k-1)\ell'}{r-1}+1$, where the last inequality holds by Lemma \ref{lem: the ineuqlaity in sec 4 (1)}. By Pigeonhole Principle, there exists  $E_{4}\in \mathcal{T}_{r-1}$ such that $v_2\in E_4$ and $E_4\cap (H\setminus\{u_1\})\not=\emptyset$. Let $u_2\in E_4\cap (H\setminus\{u_1\})$ and set $\{u_1,u_2,u_3,\ldots,u_{\ell'-1}\}\subseteq H$.
  For $i=1,2$, choose $x_{i}\in S_{i}\setminus \{v_i\}.$ Since $|BCN(H)|\ge \alpha n$, there are  $\{\omega_1,\omega_3,\omega_4,\ldots,\omega_{\ell'-1}\}$ and $\{E_{1,1},E_{1,3},E_{3,3},E_{3,4},\ldots,E_{\ell'-1,\ell'1-},E_{\ell'-1,2}\}$ such that $$\mathcal{P}:=\{x_{1},E_1,v_1,E_3,u_1,E_{1,1},\omega_1,E_{1,3},u_3,\ldots,u_{\ell'-1},E_{\ell'-1,\ell'-1},\omega_{\ell'-1},E_{\ell'-1,2},u_{2},E_4,v_2,E_2,x_{2}\}$$ forms a $\ber P_{\ell}$ and  $DV_{\mh}(\mathcal{P})\cap V(\mg)=\{u_1,u_2,\ldots,u_{\ell'-1}\}$, a contradiction with Lemma \ref{pro: intersecting A with ell}.

Suppose $\mathcal{U}^{X^+}\not=\emptyset$ when $\ell$ is odd. Let $S\in \mathcal{U}^{X^+}$, $v\in S\cap X^+$ and $S\subseteq E$, where $E\in \mh$. Since $v\in X^+$, there exist  $E'\in \mathcal{T}_{r-1}$ and  $u\in V(\mg)$ with $\{u,v\}\subseteq E'$. Obviously, $E\neq E'$. By Lemma \ref{pro: find berge P2l from mg}, there is a copy of $\ber P_{2\ell'-1}$ such that $u$ is an end vertex of the path. Together with $E'$ and $E$, there is a copy of $\ber P_{2\ell'+1}$, denoted as $\mathcal{P}$. Then there is  a copy of $\ber P_\ell$, say $\mathcal{P}'$, such that $V(\mathcal{P}') \subseteq V(\mathcal{P})$ but $|DV_\mh(\mathcal{P}')\cap V(\mg)|= \ell'-1$, a contradiction with Lemma \ref{pro: intersecting A with ell}.
\end{proof}
\par
Recall $\mathcal{U}:=\mathcal{U}_2\cup \mathcal{U}_{3+}$. By Lemma \ref{lem: no Berge P3}, the hyperedges in $\mathcal{U}$ are pairwise disjoint.
Since $\mathcal{U}[X\setminus X^+]$ is pairwise disjoint, we have $|\mathcal{U}[X\setminus X^+]|\leq \frac{|X|-|X^+|}{2}$.
By Lemma \ref{lem: U X+ is small}, $|\mathcal{U}^{X^+}|\leq 1$. If $|\mathcal{U}^{X^+}|= 1$, then we let $\mathcal{U}^{X^+}=\{S_0\}$ for some $S_0\in\mathcal{U}$ and define $\mathcal{S}_{0}=\{E\in \mathcal{T}_{t_0}|S_0\subseteq E\}$, where $t_0=r-|S_{0}|$.  Let $|X\setminus X^+|=x$. Then $|X^+|=|X|-x$.
\par
Let
\begin{equation*}
    \begin{aligned}
        &E[X^{+},V(\mg)]=\{E\in E(X,V(\mg)): E\cap X^{+}\neq \emptyset, E\cap V(\mg)\neq \emptyset\},\\
        &E[X\setminus X^{+},V(\mg)]=\{E\in E(X,V(\mg)):(E\cap X)\subseteq (X\setminus X^{+}), E\cap V(\mg)\neq \emptyset\}.
    \end{aligned}
\end{equation*}

\begin{lemma}\label{lem: U X+ upper}
   We have $$|\mh[X]|+|E[X\setminus X^{+},V(\mg)]|\leq \left(\binom{k\ell'-1}{r-1}-\frac{1}{2}\binom{k\ell'}{r-2}\right)x+\binom{k\ell'-1}{r-2}\left\lfloor\frac{x}{2}\right\rfloor.$$
\end{lemma}
\begin{proof} By Lemma \ref{lem: U X+ is small},
$|\mathcal{U}^{X+}|\leq 1$. Note that $\left|\mathcal{T}_{r-2}^+\right|\leq \binom{|V(\mg)|}{r-2}\left|\mathcal{U}_2^+\right|=\binom{k\ell'-1}{r-2}\left|\mathcal{U}_2^+\right|$.
We consider two cases.

\textbf{Case 1:} $\mathcal{U}^{X+}=\emptyset$ or $|\mathcal{U}^{X+}|= 1$ but $t_0\neq r-2$.

In this case, $\mathcal{U}_2^+\subseteq X\setminus X^{+}$. Then $|\mathcal{U}_{2}^{+}|\leq \left\lfloor\frac{x}{2}\right\rfloor$ by the hyperedges in $\mathcal{U}$ being pairwise disjoint.
When $t_0\neq r-2$, let $\mathcal{T}_{t_0}^\star=\mathcal{T}_{t_0}\setminus\mathcal{S}_0$. Set $\mathcal{T}'=(\cup_{t\in[0,r-3]\setminus\{t_0\}}\mathcal{T}_t)\cup \mathcal{T}_{t_0}^*\cup\mathcal{T}_{r-2}^-$ if $t_0\neq r-2$, and $\mathcal{T}'=(\cup_{t\in[0,r-3]}\mathcal{T}_t)\cup \mathcal{T}_{r-2}^{-}$ if $\mathcal{U}^{X+}=\emptyset$. By Lemma \ref{pro: degree of 3-set is at most 2},  $|\mathcal{T}'|\leq 2|\mathcal{U}[X\setminus X^+]|\leq|X|-|X^+|=x.$ So
\begin{equation*}
    \begin{aligned}
        &|\mh[X]|+|E[X\setminus X^+,V(\mg)]|\\
        &\leq \left|\mathcal{T}'\right|+\left|\mathcal{T}_{r-2}^+\right|+\left(\binom{k\ell'-1}{r-1}-\frac{1}{2}\binom{k\ell'}{r-2}-1\right)x\\
        &\leq  \left(\binom{k\ell'-1}{r-1}-\frac{1}{2}\binom{k\ell'}{r-2}\right)x+\binom{k\ell'-1}{r-2}\left|\mathcal{U}_2^+\right|\\
        &\leq  \left(\binom{k\ell'-1}{r-1}-\frac{1}{2}\binom{k\ell'}{r-2}\right)x+\binom{k\ell'-1}{r-2}\left\lfloor\frac{x}{2}\right\rfloor.
    \end{aligned}
\end{equation*}
\par\noindent
\textbf{Case 2:}
 $|\mathcal{U}^{X+}|= 1$ and $t_0=r-2$.
\par
 Let $\mathcal{T}_{r-2}^{\star-}=\mathcal{T}_{r-2}^-\setminus\mathcal{S}_0$, $\mathcal{T}_{r-2}^{\star+}=\mathcal{T}_{r-2}^+\setminus\mathcal{S}_0$ and $\mathcal{U}_2^{\star+}=\{U\in \mathcal{U}_{2}^{+}: U\subseteq X\setminus X^{+}\}$. Then $|\mathcal{U}_2^{\star +}|\leq \left\lfloor\frac{x}{2}\right\rfloor$.  By Lemma \ref{pro: degree of 3-set is at most 2}, we have
\begin{equation*}\label{eq: when t_0=r-2}
    \left|\left(\bigcup_{t\in[0,r-3]}\mathcal{T}_t\right)\cup \mathcal{T}_{r-2}^{\star-}\right|\leq 2|\mathcal{U}[X\setminus X^+]|\leq|X|-|X^+|=x.
\end{equation*}So we have
\begin{equation*}
    \begin{aligned}
        &|\mh[X]|+|E[X\setminus X^+,V(\mg)]|\\
        &\leq \left|\left(\bigcup_{t\in[0,r-3]}\mathcal{T}_t\right)\bigcup\mathcal{T}_{r-2}^{\star-}\right|+\left|\mathcal{T}_{r-2}^{\star+}\right|+\left(\binom{k\ell'-1}{r-1}-\frac{1}{2}\binom{k\ell'}{r-2}-1\right)x\\
        &\leq  \left(\binom{k\ell'-1}{r-1}-\frac{1}{2}\binom{k\ell'}{r-2}\right)x+\binom{k\ell'-1}{r-2}|\mathcal{U}_2^{\star+}|\\
        &\leq\left(\binom{k\ell'-1}{r-1}-\frac{1}{2}\binom{k\ell'}{r-2}\right)x+\binom{k\ell'-1}{r-2}\left\lfloor\frac{x}{2}\right\rfloor.
    \end{aligned}
\end{equation*}
\par\noindent
\end{proof}

\begin{lemma}\label{X+V(mg)}
We have
$$    |E[X^{+},V(\mg)]|\leq\left\{
\begin{array}{ll}
\binom{k\ell'-1}{r-1}(|X|-x) & \mbox{ if ~$\mathcal{U}^{X^+}=\emptyset$,}\\
\binom{k\ell'-1}{r-1}(|X|-x)+2 & \mbox{ if ~$\mathcal{U}^{X^+}\not=\emptyset$ and $t_0 \neq r-2$,}\\
\binom{k\ell'-1}{r-1}(|X|-x)+\binom{k\ell'-1}{r-2} & \mbox{ if ~$\mathcal{U}^{X^+}\not=\emptyset$ and $t_0= r-2$.}
\end{array}
\right.$$
\end{lemma}
\begin{proof}
If $\mathcal{U}^{X^+}=\emptyset$, by Lemma \ref{4.12}, we have $$|E[X^+,V(\mg)]|\leq \binom{k\ell'-1}{r-1}(|X|-x). $$
Now we will finish the proof by considering the following two cases.

\noindent\textbf{Case 1:} $\mathcal{U}^{X^+}\not=\emptyset$ and $t_0\neq r-2$.

By Lemma \ref{pro: degree of 3-set is at most 2}, we have $|\mathcal{S}_0|\leq 2$. Thus
$$|E[X^{+},V(\mg)]|\leq \binom{k\ell'-1}{r-1}(|X|-x)+2.$$
\par\noindent
\textbf{Case 2:} $\mathcal{U}^{X^+}\not=\emptyset$ and
$t_0= r-2$.

In this case, we have $|\mathcal{S}_0|\leq \binom{k\ell'-1}{r-2}$. Thus
$$|E[X^{+},V(\mg)]|\leq \binom{k\ell'-1}{r-1}(|X|-x)+\binom{k\ell'-1}{r-2}.$$
\end{proof}

\vskip.2cm
Now we are going to prove our main result.

\noindent\emph{Proof of Theorem \ref{thm: berge linear forest}} Fellow the notions and proofs of the lemmas before,
we have  $|X|+|Y|+|V(\mg)|=n$. We first have the following claim.
\begin{claim}\label{claim: finail-claim}
    If $\ell$ is even, then $\mathcal{U}^{X^+}\not=\emptyset$ and $t_0= r-2$.
\end{claim}

    \noindent{\bf Proof of Claim \ref{claim: finail-claim}}
Suppose $\mathcal{U}^{X^+}=\emptyset$ when $\ell$ is even or
$t_0\not= r-2$ when $\mathcal{U}^{X^+}\not=\emptyset$. By Lemma \ref{X+V(mg)}, $|E[X^{+},V(\mg)]|\leq \binom{k\ell'-1}{r-1}(|X|-x)+2$.
 By (\ref{eq4-1}), (\ref{eq: number of edges 1}), (\ref{eq: E(A,Y)}), Lemmas \ref{lem: U X+ upper} and \ref{X+V(mg)}, we have
\begin{equation*}
    \begin{aligned}
        |\mh|\leq& |E[X^+,V(\mg)]|+|\mh[X]|+|E[X\setminus X^+,V(\mg)]|+|E(X,Y)|+|\mh[Y]|+|\mg|+|E[V(\mg),Y]|\\
        \leq &\binom{k\ell'-1}{r-1}(|X|-x)+2+\left(\binom{k\ell'-1}{r-1}-\frac{1}{2}\binom{k\ell'}{r-2}\right)x+\binom{k\ell'-1}{r-2}\left\lfloor\frac{x}{2}\right\rfloor\\
        +& \left(\frac{1}{2\ell'}\binom{2\ell'}{r}+2 \right)|Y|+\binom{k\ell'-1}{r}\\
        = & \binom{k\ell'-1}{r-1}|X|+\binom{k\ell'-1}{r}+2+\left(\frac{1}{2\ell'}\binom{2\ell'}{r}+2 \right)|Y|
        + \binom{k\ell'-1}{r-2}\left\lfloor\frac{x}{2}\right\rfloor-\frac{1}{2}\binom{k\ell'}{r-2}x\\
        < & \binom{k\ell'-1}{r-1}(n-k\ell'+1)+\binom{k\ell'-1}{r}+\binom{k\ell'-1}{r-2},
    \end{aligned}
\end{equation*}a contradiction with (\ref{3-2}).\q

By Claim \ref{claim: finail-claim}, (\ref{eq: number of edges 1}), (\ref{eq: E(A,Y)}), Lemmas \ref{lem: U X+ upper} and \ref{X+V(mg)}, we have
\begin{equation*}
    \begin{aligned}
        |\mh|\leq& |E[X^+,V(\mg)]|+|\mh[X]|+|E[X\setminus X^+,V(\mg)]|+|E(X,Y)|+|\mh[Y]|+|\mg|+|E[V(\mg),Y]|\\
        \leq &\binom{k\ell'-1}{r-1}(|X|-x)+\mathbb{I}_\ell\cdot\binom{k\ell'-1}{r-2}+\left(\binom{k\ell'-1}{r-1}-\frac{1}{2}\binom{k\ell'}{r-2}\right)x+\binom{k\ell'-1}{r-2}\lfloor\frac{x}{2}\rfloor\\
        +& \left(\frac{1}{2\ell'}\binom{2\ell'}{r}+2 \right)|Y|+\binom{k\ell'-1}{r}\\
        = & \binom{k\ell'-1}{r-1}|X|+\binom{k\ell'-1}{r}+\mathbb{I}_\ell\cdot\binom{k\ell'-1}{r-2}+\left(\frac{1}{2\ell'}\binom{2\ell'}{r}+2 \right)|Y|\\
        +& \binom{k\ell'-1}{r-2}\left\lfloor\frac{x}{2}\right\rfloor-\frac{1}{2}\binom{k\ell'}{r-2}x\\
        \leq & \binom{k\ell'-1}{r-1}(n-k\ell'+1)+\binom{k\ell'-1}{r}+\mathbb{I}_\ell\cdot\binom{k\ell'-1}{r-2}
    \end{aligned}
\end{equation*}
 where $\mathbb{I}_\ell=1$ if $\ell$ is even, and $\mathbb{I}_\ell=0$ otherwise.
Thus we finished the proof. \QEDopen

\section{Concluding remarks}
In Theorem \ref{thm: berge linear forest},
the range of $r$ is restricted by $\ell$ or $\lf \frac{\ell+1}{2}\rf$. But we believe that the results still hold in a wider range.
\begin{conjecture}
    For integers $\ell,r$ and $k\geq 2$, there exists $N_{\ell,r,k}$ such that for $n>N_{\ell,r,k}$, and $2\leq r\leq k\lf\frac{\ell+1}{2}\rf$,
    $$\mathrm{ex}_r(n,\ber kP_\ell)=\binom{k\ell'-1}{r-1}(n-k\ell'-1+1)+\binom{k\ell'-1}{r}+\mathbb{I}_r\binom{k\ell'-1}{r-2},$$
    where $\ell'=\lf\frac{\ell+1}{2}\rf.$
\end{conjecture}
The lower bound is similar to the construction in Section 1. Moreover, with a similar construction, we can further study the Tur\'an number of a linear forest.
\begin{conjecture}
    Let $\mathcal{P}=P_{\ell_1}\cup\dots\cup P_{\ell_k}$ be a linear forest where $k\geq 2$ and $\ell_k\geq \dots,\geq \ell_1\geq 3$. Let $\ell_i=\lf\frac{\ell_i+1}{2}\rf$. When $2\leq r\leq \sum_{i=1}^k\ell_i'$,
    $$\mathrm{ex}_r(n,\ber\mathcal{P})=\max\{\mathrm{ex}_r(n,\ber P_{\ell_k}),f(n,\ell_i,r)\},$$
    where
    $$f(n,\ell_i,r)=\binom{\sum_{i=1}^k\ell_i'-1}{r-1}(n-\sum_{i=1}^k\ell_i'+1)+\binom{\sum_{i=1}^k\ell_i'-1}{r}+\mathbb{I}_{\Pi\ell_i}\cdot\binom{\sum_{i=1}^k\ell_i-1}{r-2},$$
    and $\mathbb{I}_{\Pi\ell_i}=0$ if $\Pi_{i=1}^k\ell_i$ is odd, and  $\mathbb{I}_{\Pi\ell_i}=1$ if $\Pi_{i=1}^k\ell_i$ is even.
\end{conjecture}

\section*{Acknowledgement}
This work is supported by the National Natural Science Foundation of China (Grant 12571372).

\section*{Declaration of competing interest}
The authors declare that they have no known competing financial interests or personal relationships that could have appeared to influence the work reported in this paper.

\section*{Data availability}
No data was used for the research described in the article.

\end{document}